\documentclass{article}
\usepackage{fancyhdr}
\usepackage{amscd}
\usepackage{graphicx}
\usepackage{amsmath}
\usepackage{amssymb}
\usepackage{amsthm}
\usepackage{mathrsfs}
\newtheorem{dfn}{Definition}[section]
\newtheorem{thm}[dfn]{Theorem}
\newtheorem{prop}[dfn]{Proposition}
\newtheorem{lem}[dfn]{Lemma}

\newtheorem{rem}[dfn]{Remark}

\newtheorem{ex}[dfn]{Example}

\newcommand{\diag}{{\rm diag}}

\numberwithin{equation}{section}

\begin{document}

\title{On blow-up solutions of differential equations\\ with Poincar\'{e}-type compactifications}

\author{Kaname Matsue\thanks{Institute of Mathematics for Industry, Kyushu University, Fukuoka 819-0395, Japan {\tt kmatsue@imi.kyushu-u.ac.jp}} $^{,}$ \footnote{International Institute for Carbon-Neutral Energy Research (WPI-I$^2$CNER), Kyushu University, Fukuoka 819-0395, Japan}
}
\maketitle

\begin{abstract}
We provide explicit criteria for blow-up solutions of autonomous ordinary differential equations.
Ideas are based on the quasi-homogeneous desingularization (blowing-up) of singularities and compactifications of phase spaces, which suitably desingularize singularities at infinity.
We derive several type of compactifications and show that dynamics at infinity is qualitatively independent of the choice of such compactifications.
We also show that hyperbolic invariant sets, such as equilibria and periodic orbits, at infinity induce blow-up solutions with specific blow-up rates.
In particular, blow-up solutions can be described as trajectories on stable manifolds of equilibria at infinity for associated vector fields.
Finally, we demonstrate blow-up solutions of several differential equations.
\end{abstract}


\section{Introduction}
The blow-up phenomenon, which describes divergence of solutions in finite time, is one of typical and essential singular behavior in dynamical systems generated by nonlinear differential equations.
There are many mathematical and physical studies of blow-up behaviors, and their concrete behavior such as blow-up profile, blow-up rate and blow-up sets are ones of central issues in studies of blow-up solutions.
\par
The difficulty for studying blow-up solutions is mainly the treatment of {\em infinity} from both mathematical and numerical viewpoints.
Scaling of solutions is often used for detecting the asymptotic profile of blow-up solutions \cite{FM2002}, which is one of the most essential approach of blow-up solutions in the category of {\em bounded} objects.
\par
An alternative way to studying blow-up solutions as those in bounded region is {\em compactification}.
This methodology embeds the original phase spaces, which is often the Euclidean spaces, into compact manifolds possibly with boundaries, and maps dynamics on the original phase space to those on compact manifolds.
Infinity is then mapped into extra points or the boundary of compact manifolds.
Such a treatment enables us to consider divergent solutions including blow-up solutions in terms of bounded solutions on compact manifolds.
Well-known compactifications are Bendixson's one, embedding of $\mathbb{R}^n$ into the unit $n$-sphere $S^n\subset \mathbb{R}^{n+1}$, and Poincar\'{e}'s one, embedding of $\mathbb{R}^n$ into the unit upper hemisphere $\{(x_1, \cdots, x_{n+1})\mid x_{n+1} > 0, \sum_{i=1}^{n+1} x_i^2 = 1\}$ (see e.g., \cite{H}).
In \cite{EG2006}, Elias and Gingold have discussed an admissible class of compactifications including the Poincar\'{e} compactification, which gives an appropriate correspondence of dynamics between on the original phase spaces and on compactified manifolds.
Such compactifications are recently applied to validating blow-up solutions of ordinary differential equations (ODEs for short) with {\em rigorous numerics}, a series of techniques so that all numerical errors are included in all numerical results by, say intervals, and that resulting numerical results contain mathematically rigorous objects \cite{TMSTMO}.
The result opens the door to studying blow-up solutions as {\em finite-time singularities} from the viewpoint of computer-assisted analysis of dynamical systems.
\par
A typical property of admissible compactifications, as well as well-known Bendixson's one, is that all transformations are {\em homogeneous}, namely, $y_i \mapsto y_i/\kappa$ for all $i$ with a positive functional $\kappa = \kappa(y)$.
Such a transformation yields the property that dynamics at infinity is dominated by the highest order term of vector fields, which means that all remaining lower-order terms have no effects at infinity.
This property may drop genuine scaling information of dynamics near infinity.
\par
On the other hand, quasi-homogeneous type compactifications are discussed in e.g., \cite{D2006, DH1999}, which aims at studying periodic orbits and dynamics including infinity related to Hilbert's 16th problem.
The quasi-homogeneous type compactifications are kinds of {\em quasi-homogeneous desingularizations}\footnote
{
In the terminology of algebraic geometry or ordinary singularity theory of dynamical systems, it is often called {\em blow-up}.
However, we shall use the terminology {\em desingularization} here to avoid any confusions to {\em blow-up solutions} of differential equations.
}
near infinity, which effectively desingularize singularities for precise analysis of dynamics around singularities.
However, most of studies using quasi-homogeneous type compactifications have concerned only with desingularized two-dimensional dynamical systems such as {\em Li\'{e}nard} equations, and the aspect of {\em finite-time singularities} or their {\em higher dimensional analogues} are not well-considered.
\par
\bigskip
Our main aim here is to describe blow-up behavior of solutions for vector fields including those which are not necessarily homogeneous near infinity from the viewpoint of dynamical systems.
We treat not only well-known quasi-homogeneous compactifications in e.g., \cite{D2006, DH1999} but also a prototype of the quasi-homogeneous analogue of admissible compactifications in \cite{EG2006}, which we shall call {\em quasi-Poincar\'{e} compactifications}, and unify the aspect of dynamics at infinity obtained in these individual compactifications. 
In other words, we can say that the qualitative information of blow-up behavior is independent of the choice of compactifications.
As primitive results, we state the stationary blow-up and periodic blow-up under generic assumptions, which corresponds to trajectories on stable manifolds of equilibria or periodic orbits at infinity for compactified vector fields, respectively.

\par
\bigskip
The rest of this paper is organized as follows.
In Section \ref{section-QHPoincare}, we introduce a new compactification called quasi-Poincar\'{e} compactification with a brief review of quasi-homogeneous desingularizations.
We also review known quasi-homogeneous type compactifications and discuss the correspondence.
In Section \ref{section-dyn-infty}, we discuss the transformation of dynamics via compactifications.
We obtain desingularized vector fields for individual compactifications, and we prove that these vector fields are topologically equivalent including infinity, which shows that qualitative properties of dynamics at infinity are independent of the choice of certain quasi-homogeneous type compactifications.
In Section \ref{section-blow-up}, we show several generic results for describing blow-up solutions in terms of global trajectories asymptotic to invariant sets at infinity.
We see that convergence of trajectories {\em with exponential rate} and {\em the in-phase property} for hyperbolic equilibria and periodic orbits at infinity induce blow-up solutions with specific asymptotic blow-up behavior.
Several demonstrations of our main results with numerical simulations are shown in Sections \ref{section-Lienard}, \ref{section-sing-shock} and \ref{section-two-fluid}.

\section{Compactifications}
\label{section-QHPoincare}

In this section, we introduce an alternative compactification of admissible ones discussed in e.g., \cite{P, H, EG2006}.
Our compactification is based on the concept of quasi-homogeneous desingularization of singularities in dynamical systems generated by vector-valued functions with an appropriate scaling at infinity.
Firstly we briefly review quasi-homogeneous vector fields.
Secondly, we introduce the new compactification called {\em quasi-Poincar\'{e} compactification} and provide several fundamental properties.
Our compactification is also a higher-dimensional alternative to Poincar\'{e}-Lyapunov discs discussed in e.g., \cite{DH1999, D2006}.

\subsection{Quasi-homogeneous vector fields}
\label{section-QH}

\begin{dfn}[Quasi-homogeneous vector fields, cf. \cite{D1993}]\rm
\label{dfn-QHvector-fields}
Let $f : \mathbb{R}^n \to \mathbb{R}$ be a smooth function.
Let $\alpha_1,\cdots, \alpha_n\geq 0$\footnote{
Quasi-homogeneity is usually defined with $\alpha_1,\cdots, \alpha_n \geq 1$.
But the natural extension to the case $\alpha_i= 0$ still makes sense, and hence we use this generalized one.
} with $(\alpha_1,\cdots, \alpha_n)\not = (0,\cdots, 0)$ be integers and $k \geq 1$.
We say that $f$ is a {\em quasi-homogeneous function of type $(\alpha_1,\cdots, \alpha_n)$ and order $k$} if
\begin{equation*}
f(R^{\alpha_1}x_1,\cdots, R^{\alpha_n}x_n) = R^k f(x_1,\cdots, x_n),\quad \forall x\in \mathbb{R}^n,\quad R\in \mathbb{R}.
\end{equation*}
\par
Next, let $X = \sum_{j=1}^n f_j(x)\frac{\partial }{\partial x_j}$ be a smooth vector field.
We say that $X$, or $f = (f_1,\cdots, f_n)$ is a {\em quasi-homogeneous vector field of type $(\alpha_1,\cdots, \alpha_n)$ and order $k+1$} if each component $f_j$ is a homogeneous function of type $(\alpha_1,\cdots, \alpha_n)$ and order $k + \alpha_j$.
\end{dfn}

For applications to general vector fields, we define the following notion.

\begin{dfn}[Homogeneity index and admissible domain]\rm
Let $\alpha = (\alpha_1,\cdots, \alpha_n)$ be a set of nonnegative integers.
Let the index set $I_\alpha$ as $I_\alpha=\{i\in \{1,\cdots, n\}\mid \alpha_i > 0\}$, which we shall call {\em the set of homogeneity indices associated with $\alpha = (\alpha_1,\cdots, \alpha_n)$}.
Let $U\subset \mathbb{R}^n$.
We say the domain $U\subset \mathbb{R}^n$ {\em admissible with respect to the sequence $\alpha$}
if
\begin{equation*}
U = \{x=(x_1,\cdots, x_n)\in \mathbb{R}^n \mid x_i\in \mathbb{R}\text{ if }i\in I_\alpha,\ (x_{j_1},\cdots, x_{j_{n-l}}) \in \tilde U\},
\end{equation*} 
where $\{j_1, \cdots, j_{n-l}\} = \{1,\cdots, n\}\setminus I_\alpha$ and $\tilde U$ is an $(n-l)$-dimensional open set.
\end{dfn}
Assumptions in Definition \ref{dfn-QHvector-fields} indicate $I_\alpha \not = \emptyset$.
The notion of asymptotic quasi-homogeneity provides a systematic validity of scalings in many practical applications.

\begin{dfn}[Asymptotically quasi-homogeneous vector fields]\rm
Let $f = (f_1,\cdots, f_n):U \to \mathbb{R}^n$ be a smooth function with an admissible domain $U\subset \mathbb{R}^n$ with respect to $\alpha$ such that $f$ is uniformly bounded for each $x_i$ with $i\in I_\alpha$, where $I_\alpha$ is the set of homogeneity indices associated with $\alpha$.
We say that $X = \sum_{j=1}^n f_j(x)\frac{\partial }{\partial x_j}$, or simply $f$ is an {\em asymptotically quasi-homogeneous vector field of type $(\alpha_1,\cdots, \alpha_n)$ and order $k+1$ at infinity} if
\begin{equation*}
\lim_{R\to +\infty} R^{-(k+\alpha_j)}\left\{ f_j(R^{\alpha_1}x_1, \cdots, R^{\alpha_n}x_n) - R^{k+\alpha_j}(f_{\alpha,k})_j(x_1, \cdots, x_n) \right\}
 = 0
 \end{equation*}
holds uniformly for $(x_1,\cdots, x_n)\in U_1$, where $f_{\alpha,k} = ((f_{\alpha,k})_1,\cdots, (f_{\alpha,k})_n)$ is a quasi-homogeneous vector field of type $(\alpha_1,\cdots, \alpha_n)$ and order $k+1$, and
\begin{equation*}
U_1 = \{x=(x_1,\cdots, x_n)\in \mathbb{R}^n \mid (x_{i_1}, \cdots, x_{i_l})\in S^{l-1},\ (x_{j_1},\cdots, x_{j_{n-l}}) \in \tilde U\},
\end{equation*}
where $\{i_1,\cdots, i_l\} = I_\alpha$.
\end{dfn}

\subsection{Quasi-Poincar\'{e} compactifications}
\label{section-definition}
Throughout successive sections, consider the (autonomous) vector field
\begin{equation}
\label{ODE-original}
y' = f(y),
\end{equation}
where $f : U \to \mathbb{R}^n$ be a smooth function with an admissible domain $U\subset \mathbb{R}^n$ with respect to $\alpha$.
Throughout our discussions, we assume that $f$ is an asymptotically quasi-homogeneous vector field of type $\alpha = (\alpha_1,\cdots, \alpha_n)$ and order $k+1 > 1$ at infinity.

\begin{dfn}[Quasi-Poincar\'{e} compactification]\rm
\label{dfn-qP}
Let $a_1,\cdots, a_n \geq 1$, and $\beta_1,\cdots, \beta_n$ be nonnegative numbers such that 
\begin{equation}
\label{LCM}
\begin{cases}
\alpha_i \beta_i \equiv c \in \mathbb{N} & \text{ if $i\in I_\alpha$}, \\
\beta_i = 0 & \text{ otherwise.}
\end{cases}
\end{equation}
Define {\em quasi-Poincar\'{e} functionals} $p(y)$ and $\kappa(y)$ as
\begin{equation*}
p(y) = p_{\alpha, \bf a}(y) := \left( \sum_{i\in I_\alpha}a_i y_i^{2\beta_i} \right)^{1/2c},\quad \kappa(y) = \kappa_{\alpha, {\bf a}}(y) := (1+p_{\alpha, {\bf a}}(y)^{2c})^{1/2c}.
\end{equation*}
where ${\bf a}$ in the subscript denotes the dependence on $\{a_i\}_{i=1}^n$.
Define the {\em quasi-Poincar\'{e} compactification of type $(\alpha_1, \cdots, \alpha_n)$} as
\begin{equation}
\label{coord-qP}
T_{qP} : \mathbb{R}^n \to \mathbb{R}^n,\quad T_{qP}(y) = x,\quad x_i := \frac{y_i}{\kappa_{\alpha, {\bf a}}(y)^{\alpha_i}}.
\end{equation}
\end{dfn}
Obviously, $x_i = y_i$ if $i\not \in I_\alpha$.
We immediately observe that $p_{\alpha,{\bf a}}(x)\to 1$ as $p_{\alpha, {\bf a}}(y)\to \infty$, and vice versa.
Therefore, the infinity in the original coordinates corresponds to a point on
\begin{equation*}
\mathcal{E} = \{x \in U \mid p_{\alpha, {\bf a}}(x) = 1\}.
\end{equation*}
We shall call the set $\mathcal{E}$ {\em the horizon}\footnote{
In the case of $I_\alpha = \{1,\cdots, n\}$, the set $\mathcal{E}$ is often called {\em the equator}, in which case $\mathcal{E} = \partial \mathcal{D}$.
}.
If no confusions arise, we drop the subscripts ${\alpha, {\bf a}}$ in the expression of $p_{\alpha, {\bf a}}$ and $\kappa_{\alpha, {\bf a}}$.
Note that the quasi-Poincar\'{e} functional $\kappa$ in the $x$-coordinate is
\begin{equation*}
\kappa(T_{qP}^{-1}(x)) = \left(1 - \sum_{j\in I_\alpha} a_j x_j^{2\beta_j}\right)^{-1/2c}. 
\end{equation*}

\begin{rem}\rm
The simplest choice of the natural number $c$ is the least common multiple of $\{\alpha_i\}_{i\in I_\alpha}$.
Once we choose such $c$, we can determine the $n$-tuples of natural numbers $\beta_1,\cdots, \beta_n$ uniquely.
The choice of natural numbers in (\ref{LCM}) is essential to desingularize vector fields at infinity, as shown below.
\end{rem}

\begin{dfn}\rm
We say that a solution orbit $y(t)$ of (\ref{ODE-original}) with the maximal existence time $(a,b)$, possibly $a = -\infty$ and $b = +\infty$, {\em tends to infinity in the direction $x_\ast \in \mathcal{E}$ associated with the quasi-Poncar\'{e} functional $p$} (as $t\to a+0$ or $b-0$) if
\begin{equation*}
p(y(t))\to \infty,\quad \left(\frac{y_1}{\kappa(y)^{\alpha_1}}, \cdots, \frac{y_n}{\kappa(y)^{\alpha_n}}\right)\to x_\ast\quad \text{ as }t\to a+0 \text{ or }b-0.
\end{equation*}
\end{dfn}

Now compute the Jacobian matrix $J$ of $T$.
Without the loss of generality, by taking permutations of coordinates if necessary, we may assume that $I_\alpha = \{1,2,\cdots, l\}$.
Direct computations yield
\begin{equation*}
\frac{\partial x_i}{\partial y_j} = 
\begin{cases}
\kappa^{-\alpha_i}\left( \delta_{ij} - \kappa^{-1} \alpha_i y_i \frac{\partial \kappa}{\partial y_j}\right) & \text{$j \in \{1,\cdots,l\}$}\\
\delta_{ij} & \text{$j \in \{l+1,\cdots, n\}$}
\end{cases}
\end{equation*}
with the matrix form
\begin{align*}
&J = \left(\frac{\partial x_i}{\partial y_j}\right)_{i,j = 1,\cdots, n} = A_{\alpha} \left( I_n - \kappa^{-1} y_\alpha (\nabla \kappa)^T\right),\\
&A_\alpha = \diag(\kappa^{-\alpha_1}, \cdots, \kappa^{-\alpha_l}, 1,\cdots, 1),\quad y_\alpha = (\alpha_1 y_1,\cdots, \alpha_l y_l, 0,\cdots, 0)^T.
\end{align*}

We follow arguments in \cite{EG2006}, for any (column) vectors $y,z\in \mathbb{R}^n$, to have 
\begin{align*}
(I_n + \beta yz^T)(I_n + \beta yz^T) &= I + (\beta + \delta)yz^T + \beta \delta yz^T yz^T\\
	&= I + (\beta + \delta + \beta \delta \langle z,y\rangle)yz^T,
\end{align*}
so $I+\delta yz^T = (I+\delta yz^T )^{-1}$ if $\delta = -\beta / (1 + \beta \langle z,y\rangle)$.
\par
In this case, we choose $\beta = -\kappa^{-1}, y = y_\alpha, z = \nabla \kappa$ and have
\begin{equation*}
\left(\frac{\partial y_j}{\partial x_i}\right) = \left(\frac{\partial x_i}{\partial y_j}\right)^{-1} =  \left( I_n - \frac{1}{\kappa - \langle y_\alpha, \nabla \kappa \rangle } y_\alpha (\nabla \kappa)^T\right)A_{\alpha}^{-1}
\end{equation*}
Now we have
\begin{equation*}
\frac{\partial \kappa}{\partial y_j} = \frac{\partial }{\partial y_j} \left(1+\sum_{i=1}^l a_i y_i^{2\beta_i} \right)^{\frac{1}{2c}} = \frac{\beta_j}{c} \left(1+\sum_{i=1}^l a_i y_i^{2\beta_i} \right)^{\frac{1}{2c}-1} a_j y_j^{2\beta_j-1}
= \frac{\beta_j a_j}{c\kappa^{2c-1}} y_j^{2\beta_j-1}
\end{equation*}
if $j\in \{1,\cdots, l\} = I_\alpha$.
Obviously, $\partial \kappa/\partial y_j = 0$ holds if $j \in \{l+1,\cdots, n\} = \{1,\cdots, n\}\setminus I_{\alpha}$.
Hence
\begin{align*}
\kappa^{2c-1}\left(\kappa - \langle y_\alpha, \nabla \kappa \rangle \right) &= \kappa^{2c-1}\left(\kappa - \sum_{j=1}^l \alpha_j y_j \frac{\beta_j a_j}{c\kappa^{2c-1}} y_j^{2\beta_j-1} \right) = \left\{ (1+p(y)^{2c}) - p(y)^{2c}\right\} > 0,
\end{align*}
which indicates that the transformation $T_{qP}$ as well as $T_{qP}^{-1}$ are $C^1$ locally bijective including $y=0$.
On the other hand, the map $T_{qP}$ maps any one-dimensional curve $y = (r^{\alpha_1}v_1, \cdots, r^{\alpha_n}v_n)$, $0\leq r < \infty$, with some fixed direction $v\in \mathbb{R}^n$, into itself. 
For continuous mappings from $\mathbb{R}$ to $\mathbb{R}$, local bijectivity implies global bijectivity.
Consequently, $T_{qP}$ is (globally) bijective.
\par
\bigskip
Summarizing these arguments, we obtain the following proposition.
\begin{prop}
\label{prop-adm}
Let $a_1,\cdots, a_n \geq 1$ be fixed.
Then the associated functional $\kappa = \kappa_{\alpha, {\bf a}}$ defining the quasi-Poincar\'{e} compactificaton $T_{qP}$ satisfies the following properties.
\begin{enumerate}
\item $T_{qP}$ is a bijection from $\mathbb{R}^n$ to $\mathcal{D} = \{x\in \mathbb{R}^n \mid x=T_{qP}(y), y\in U, p(x) < 1\}$.
\item We have
\begin{description}
\item[(A0)] $\kappa(y) > p(y)$ for all $y\in \mathbb{R}^n$,
\item[(A1)] $\kappa(y) \sim p(y)$ as $p(y)\to \infty$\footnote{
\lq\lq $F(\eta) \sim G(\eta)$ as $\eta\to \infty$\rq\rq denotes that $\lim_{\eta\to \infty} F(\eta)/G(\eta) = C\not = 0$.
}.
\item[(A2)] $\nabla \kappa(y) = ((\nabla \kappa(y))_1, \cdots, (\nabla \kappa(y))_n)$ satisfies
\begin{equation*}
(\nabla \kappa(y))_i \sim \frac{a_i}{\alpha_i}\frac{y_i^{2\beta_i-1}}{p(y)^{2c-1}}\quad \text{ as } p(y)\to \infty \text{ if }i\in I_\alpha,\quad (\nabla \kappa(y))_i \equiv 0 \text{ otherwise}.
\end{equation*}
\item[(A3)] Letting $y_\alpha = (\alpha_1 y_1,\cdots, \alpha_n y_n)^T$ for $y\in \mathbb{R}^n$, we have $\langle y_\alpha, \nabla \kappa \rangle = p(y) < \kappa(y)$ holds for any $y\in \mathbb{R}^n$.
\end{description}
\item $T_{qP}(\mathbb{R}^n)$ is extended continuously onto $\overline{\mathcal{D}}$, in particular, onto $\mathcal{E}$.
\end{enumerate}
\end{prop}

\begin{proof}
1. See discussions above.
\par
2. (A0) and (A1) follow from the definition.
(A2) follows from direct calculations of $\nabla \kappa$ and (A1).
(A3) follows from direct computations.
\par
3. 
For any sequences $\{y_k\}_{k\geq 1}$ which tend to infinity in the direction $x_\ast$, the definition $\lim_{k\to \infty}T_{qP}(y_k) \equiv x_\ast$ makes sense and shows the continuous extension of $T_{qP}(\mathbb{R}^n)$ onto $\mathcal{D}\cup \{x_\ast\}$, since $x_k \equiv T_{qP}(y_k) \to x_\ast$ and $p(y_k)\to \infty$ as $k\to \infty$; namely, $p(x_\ast) = 1$ and $x_\ast\in \overline{\mathcal{D}}$. 
Since each $x_\ast\in \partial \mathcal{D}$ is an accumulation point, there is a sequence $\{x_j = (x_{ij})^T\}_{j\geq 1}$ converging to $x_\ast$.
Letting $y_j \equiv (y_{ij})^T$ with $y_{ij} = \kappa^{\alpha_i}x_{ij}$, $i=1,\cdots, n$, for such a sequence, $\{y_j\}_{j\geq 1}$ tends to infinity in the direction $x_\ast$.
This fact shows that $T_{qP}(\mathbb{R}^n)$ is extended continuously onto $\overline{\mathcal{D}}$ and completes the proof.
\end{proof}

\begin{rem}\rm
Four properties (A0) $\sim $ (A3) in Proposition \ref{prop-adm} will play central roles in the theory of, which will be called, {\em quasi-homogeneous compactifications} and associated dynamics. 
Indeed, in the case of {\em homogeneous} compactifications, namely $\alpha_1 = \cdots = \alpha_n = \beta_1 = \cdots = \beta_n = 1$ and $a_1=\cdots = a_n = 1$, these conditions describe {\em admissibility} of compactifications \cite{EG2006}, which play central roles to dynamics at infinity.
The (homogeneous) Poincar\'{e} compactification (e.g., \cite{H, P}) is the prototype of other admissible compacifications such as parabolic ones (e.g., \cite{EG2006, TMSTMO}), and hence quasi-Poincar\'{e} compactifications with Proposition \ref{prop-adm} will be the prototype of compactifications with quasi-homogeneous desingularizations.
\end{rem}

\par
As the homogeneous version, quasi-Poincar\'{e} compactifications have the following geometric aspect.
First, regard the original phase space $\mathbb{R}^n$ as the subspace $\mathbb{R}^n \times \{1\}$ in $\mathbb{R}^{n+1}$.
For any points $M=(y,1)\in \mathbb{R}^n \times \{1\}$, there is one-to-one correspondence between $M$ and the point on the quasi-hemisphere
\begin{equation*}
\mathcal{H}_\alpha = \{(x,\zeta)\in \mathbb{R}^{n+1} \mid \zeta > 0,\ p(x)^{2c} + \zeta^{2c} = 1\},
\end{equation*}
as the intersection of $\mathcal{H}$ and the curve
\begin{equation}
\label{curve}
C_\alpha: \mathbb{R}^{n+1}\to \mathbb{R}^{n+1},\quad C_\alpha(y,\zeta) = \left(\zeta^{\alpha_1}y_1,\cdots, \zeta^{\alpha_n}y_n, \zeta \right)
\end{equation}
with endpoints $(0,0)\in \mathbb{R}^{n+1}$ and $M$.
The intersection is given by the point $C_\alpha(y, \zeta)$ on the curve $C_\alpha$ satisfying
\begin{equation*}
\sum_{i=1}^n (\zeta^{\alpha_i} a_i y_i)^{2\beta_i} + \zeta^{2c} = \zeta^{2c}\left(\sum_{i=1}^n a_i y_i^{2\beta_i} + 1\right) = 1.
\end{equation*}
The explicit representation of the point is
\begin{equation*}
x = \left(\zeta^{\alpha_1}y_1, \cdots, \zeta^{\alpha_n}y_n\right) \equiv \left(\frac{y_1}{\kappa(y)^{\alpha_1}}, \cdots, \frac{y_n}{\kappa(y)^{\alpha_n}} \right),\quad \zeta = \frac{1}{(1+p(y)^{2c})^{1/2c}} \equiv \frac{1}{\kappa(y)}.
\end{equation*}
The quasi-Poincar\'{e} compactification $T_{qP}$ is thus given by the projection of the above intersection point onto $\mathbb{R}^n$.
This geometric representation of $T_{qP}$ gives its bijectivity stated in Proposition \ref{prop-adm}.

\begin{rem}\rm
In the case of (homogeneous) Poincar\'{e} and other admissible, homogeneous compactifications (\cite{EG2006, H, P}), the curve $C_\alpha$ is given by the line segment with endpoints $(0,0), (y,1)\in \mathbb{R}^{n+1}$.
See Figure \ref{fig-quasi-poincare}.
\end{rem}

\begin{figure}[htbp]\em
\begin{minipage}{0.5\hsize}
\centering
\includegraphics[width=7.0cm]{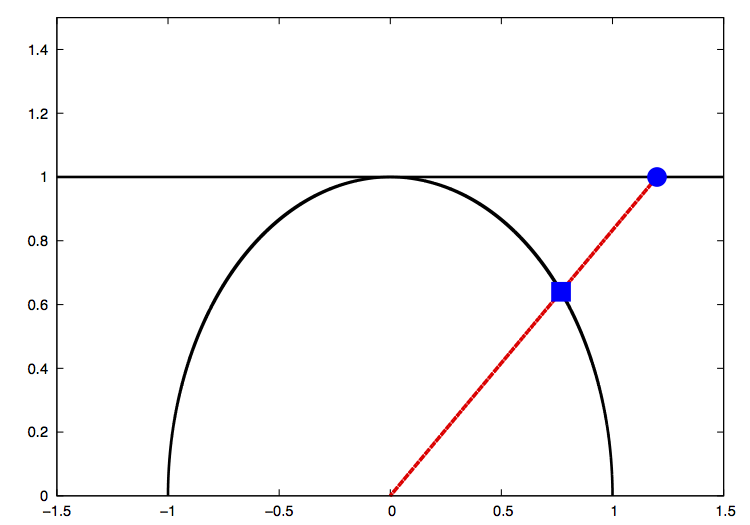}
(a)
\end{minipage}
\begin{minipage}{0.5\hsize}
\centering
\includegraphics[width=7.0cm]{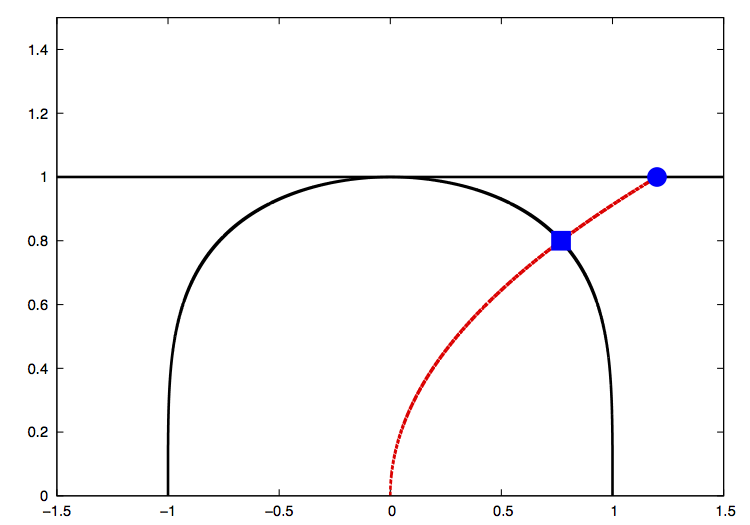}
(b)
\end{minipage}
\caption{Poincar\'{e} and quasi-Poincar\'{e} compactifications with type $(2,1)$}
\label{fig-quasi-poincare}
Both figures describe the slice $\{y_2=0\}$ of surfaces (a) : $\mathcal{H} = \{(y_1,y_2,\zeta)\mid \zeta > 0,\ y_1^2 + y_2^2 + \zeta^2 = 1\}$, and (b) : $\mathcal{H}_\alpha = \{(y_1,y_2,\zeta)\mid \zeta > 0,\ y_1^2 + y_2^4 + \zeta^4 = 1\}$.
\par
(a) : Poincar\'{e} compactification.
(b) : Quasi-Poincar\'{e} compactification with type $(2,1)$.
\par
In both figures, the black curves describe the original phase space $\mathbb{R}^n\times \{1\}\subset \mathbb{R}^{n+1}$ (straight line) and the quasi-hemispheres $H_\alpha$ (curve).
In the case of (a), the type $\alpha$ is chosen to be $(1,1)$.
Blue round points show the points in the original phase space and squared points show the intersection between $H_\alpha$ and the curve $C_\alpha$ defined by (\ref{curve}) colored by red.
The projection of squared points onto the original phase space (black straight line) are the images of (quasi-)Poincar\'{e} compactifications.
\end{figure}

\subsection{Directional and intermediate compactifications}
\label{section-coord}

There are several other coordinates which are used in preceding works (e.g., \cite{DH1999, D2006}) for studying dynamics at infinity, many of which are considered {\em locally} near infinity.
In this section, we discuss such compactifications and compare with quasi-Poincar\'{e} compactifications.

\begin{dfn}[Directional compactification]\rm
\label{dfn-directional}
Let the type $\alpha = (\alpha_1, \cdots, \alpha_n)\in \mathbb{Z}^n_{\geq 0} \setminus \{(0,\cdots, 0)\}$ be fixed.
Define a {\em directional compactification of type $\alpha$} as the transformation $T_{\bf h}:(y_1,\cdots, y_n)\mapsto (s,\theta_1,\cdots, \theta_{n-1})$ given by
\begin{equation}
\label{dfn-dir-cpt}
(y_1,\cdots, y_n) = \left(\frac{h_1(\theta_1,\cdots, \theta_{n-1})}{s^{\alpha_1}}, \cdots, \frac{h_n(\theta_1,\cdots, \theta_{n-1})}{s^{\alpha_n}}\right),
\end{equation}
where ${\bf h} = (h_1,\cdots, h_n)$ be functions defined on an $(n-1)$-dimensional (not necessarily compact) smooth manifold $M$\footnote{
The $M$ is usually assumed to be $\mathbb{R}^{n-1}$, $S^{n-1}$, $T^{n-1}$ or an open subset of them.
} parameterized by $(\theta_1,\cdots, \theta_{n-1})$ such that
\begin{description}
\item[(Dir1)] the set $E = \bigcup_{s\geq 0}\left(\{s\} \times {\bf h}(M) \right)$ forms a fiber bundle over $[0,\infty)$ whose fiber $\{s\} \times {\bf h}(M)$ is diffeomorphic to $M$.
Moreover, there is an open set $U\subset \mathbb{R}^n$ such that $T_{\bf h}$ maps $U$ into $ \bigcup_{s> 0}\left(\{s\} \times {\bf h}(M) \right)$ diffeomorphically.
\item[(Dir2)] $\sum_{i\in I_\alpha} a_i h_i(\theta_1,\cdots, \theta_{n-1})^{2\beta_i} \geq c_h > 0$ for all $\theta_1, \cdots, \theta_{n-1}$;
\item[(Dir3)] there is a smooth function $e$ such that $s = e(y)$ and that $e(y) \sim \kappa(y)^{-1}$ as $s\to 0$, which is locally uniform in $M$;
\item[(Dir4)] the matrix
\begin{equation}
\label{A-theta}
A = A(\theta_1,\cdots, \theta_{n-1}) := \begin{pmatrix}
\alpha_1 h_1 & \frac{\partial h_1}{\partial \theta_1} & \cdots & \frac{\partial h_1}{\partial \theta_{n-1}} \\
\alpha_2 h_2 & \frac{\partial h_2}{\partial \theta_1} & \cdots & \frac{\partial h_n}{\partial \theta_{n-1}} \\
\vdots & \cdots & \cdots & \vdots\\
\alpha_n h_n & \frac{\partial h_n}{\partial \theta_1} & \cdots & \frac{\partial h_n}{\partial \theta_{n-1}}
\end{pmatrix}
\end{equation}
is invertible and smooth on $M$.
\end{description}
We shall call the hypersurface $\mathcal{E}=\{(s,\theta_1,\cdots,\theta_{n-1})\mid s=0 \}$ {\em the horizon}.
\end{dfn}

The following examples show that several well-known quasi-homogeneous type compactifications are included in the above notion.

\begin{ex}\rm
Compactifications of quasi-homogeneous type for studying dynamics at infinity are applied to e.g., polynomial Li\'{e}nard equations of type $(m,n)\in \mathbb{N}^2$ (e.g., \cite{DH1999, D2006}):
\begin{equation}
\label{Lienard}
\begin{cases}
\dot x = y, & \\
\dot y = -(\epsilon x^m + \sum_{k=0}^{m-1}a_k x^k) - y(x^n + \sum_{k=0}^{n-1}b_k x^k), & 
\end{cases}
\end{equation}
where $\epsilon = \pm 1$ if $m\not = 2n+1$, and $\epsilon \in \mathbb{R}\setminus \{0\}$ if $m=2n+1$.
Dynamics of (\ref{Lienard}) near infinity is considered via the transformations such as
\begin{equation}
\label{loc-coord-PL}
(x,y)\mapsto (s,u),\quad x = \pm 1/s,\quad y = u / s^{n+1},
\end{equation}
which defines the vector field in $(s,u)$-coordinates. 
The whole domains of $(s,u)$-coordinates are often called {\em Poincar\'{e}-Lyapunov disks}, or {\em PL-disks} for short.
Infinity corresponds to $\{s=0\}$, and the other component $u$ locally determines the direction of infinity.
The power of $s$ depends on the choice of type $(m,n)$, which corresponds to the choice of $\{\beta_i\}_{i=1}^n$ associated with $\{\alpha_i\}_{i=1}^n$ in quasi-Poincar\'{e} compactifications.
The signature $\pm$ in (\ref{loc-coord-PL}) corresponds to transformations in {\em positive or negative} $x$-directions.
In other words, the set $U$ in Definition \ref{dfn-directional} is chosen as $\{x>0\}$ or $\{x<0\}$ corresponding to the signature in (\ref{loc-coord-PL}).
This fact indicates that we need multiple charts on PL-disks for complete studies of dynamics including infinity.
On the other hand, quasi-Poincar\'{e} compactifications require {\em only one chart} for various treatments; namely, the coordinate $(x_1,\cdots, x_n)$ in (\ref{coord-qP}) determines the {\em global chart} on $\overline{\mathcal{D}}$.
\par
\bigskip
The change of coordinates (\ref{loc-coord-PL}) is also regarded as local coordinate system of quasi-Poincar\'{e} compactifications.
Indeed, consider the curve $C_\alpha\subset \mathbb{R}^{n+1}$ given in (\ref{curve}).
To be simplified, assume that $I_{\alpha} = \{1,\cdots, n\}$.
Then we obtain a transformation given by
\begin{equation}
\label{intersection}
(y_1,\cdots, y_n, \kappa(y)) = \left(s^{-\alpha_1} x_1, \cdots, s^{-\alpha_{i-1}} x_{i-1}, \pm s^{-\alpha_i},  s^{-\alpha_{i+1}} x_{i+1}, \cdots, s^{-\alpha_n} x_n, s^{-1} \right),
\end{equation}
which attains the intersection of $C_\alpha$ and the $n$-dimensional upper-half hyperplane given by $\{x_i = \pm 1, s\geq 0\}$ centered at $p\in \partial \mathcal{H}_\alpha$ with $p = (x,0)\in \mathbb{R}^{n+1}$, where $x_j = \pm \delta_{ij}$ for $j=1,\cdots, n$.
The upper-half hyperplane is exactly the (upper-half) tangent space $T_p \overline{\mathcal{H}_\alpha}$ of $\overline{\mathcal{H}_\alpha}$ at $p$, and hence the transformation (\ref{intersection}) provides the local coordinate $(x_1,\cdots, x_{i-1}, s, x_{i+1}, \cdots, x_n)$ on $T_p \overline{\mathcal{H}_\alpha}$.
The local coordinate can be considered as a higher-dimensional analogue of (\ref{loc-coord-PL}).
\end{ex}

\begin{ex}\rm
Consider the case $I_\alpha \not = \{1,\cdots, n\}$; namely, $\alpha_i = 0$ for some $i$.
In such a case, the $i$-th component of $T_{\bf h}$ should be identity and hence $y_i$ is chosen as one of variables $\theta_1,\cdots,\theta_{n-1}$.
For example, consider directional compactification of type $\alpha = (1,2,0)$ in $\mathbb{R}^3$ as in (\ref{loc-coord-PL}).
Typical one would be $(y_1, y_2, y_3) = (1/s, u/s^2, y_3)$, in which case the paraemterization of $M=\mathbb{R}^2$ is $(\theta_1, \theta_2) = (u,y_3)$ and functions $\{h_i\}$ are given by $h_1 = 1, h_2 = u, h_3 = y_3$.
In particular, the characterization of directional compactifications still makes sense for $\alpha_i = 0$.
\end{ex}

\begin{ex}\rm
In two-dimensional problems, there is an alternative choice of coordinates using {\em $(1,l)$-trigonometric functions}, which is often called {\em quasi-polar coordinates}.
Let ${\rm Cs}\theta$ and ${\rm Sn}\theta$ (see e.g. \cite{DH1999, D2006} and references therein for detailed properties) be analytic functions given by the solutions of the following Cauchy problem:
\begin{equation*}
\frac{d}{d\theta}{\rm Cs}\theta = -{\rm Sn}\theta,\quad \frac{d}{d\theta}{\rm Sn}\theta = {\rm Cs}^{2l-1}\theta,\quad 
\begin{cases}
{\rm Cs}0 = 1 &\\
{\rm Sn}0 = 0 &
\end{cases}.
\end{equation*}
These functions satisfy 
\begin{equation}
\label{trigonometric}
{\rm Cs}^{2l}\theta + l {\rm Sn}^{2}\theta = 1\quad \text{ for all } \theta,
\end{equation}
and both ${\rm Cs}\theta$ and ${\rm Sn}\theta$ are $T$-periodic with
\begin{equation*}
T = T_{1,l} = \frac{2}{\sqrt{2}} \int_0^1 (1-t)^{-1/2}t^{(1-2l)/2l}.
\end{equation*}
The $(1,l)$-quasi-polar coordinate $(r,\theta)$ is given by
\begin{equation}
\label{coord-polar}
y_1 = \frac{{\rm Cs}\theta}{s},\quad y_2 = \frac{{\rm Sn}\theta}{s^l},
\end{equation}
which follows from (\ref{trigonometric}) that $y_1^{2l} + l y_2^2 = s^{-2l}$.
In particular, the parameterization surface $M$ is $S^1$, $\theta_1 = \theta$, $h_1 = {\rm Cs}\theta$ and $h_2 = {\rm Sn}\theta$.
The quasi-polar coordinate $(s,\theta)$ makes sense except the origin $(y_1, y_2) = (0,0)$ in the original coordinate.
In other words, the set $U$ in Definition \ref{dfn-directional} is chosen as $\{(y_1, y_2) \not = (0,0)\}$.
\end{ex}

The following compactification is an auxiliary one, which connects dynamics via quasi-Poincar\'{e} and directional compactifications.

\begin{dfn}[Intermediate compactification]\rm
\label{dfn-intermediate}
Let the type $\alpha = (\alpha_1, \cdots, \alpha_n)$ be fixed so that $\alpha_i \geq 0$ and $\alpha \not = (0,\cdots, 0)$.
Define an {\em intermediate compactification $T_{{\bf h}, int}$ of type $\alpha$ associated with the directional compactification $T_{\bf h}$} as the transformation
\begin{equation}
\label{dfn-inter-cpt}
(y_1,\cdots, y_n) = \left(\frac{h_1(\theta_1,\cdots, \theta_{n-1})}{R^{\alpha_1/2c}}, \cdots, \frac{h_n(\theta_1,\cdots, \theta_{n-1})}{R^{\alpha_n/2c}}\right),
\end{equation}
where ${\bf h} = (h_1,\cdots, h_n)$ is the diffeomorphic functions determining $T_{\bf h}$, such that (Dir1) - (Dir4) in Definition \ref{dfn-directional} replacing $s$ in (Dir1) and \lq\lq $s=e(y)$, $e(y)\sim \kappa(y)^{-1}$ \rq\rq in (Dir3) by $R$ and  \lq\lq $R=e(y)$, $e(y)\sim \kappa(y)^{-2c}$ \rq\rq, respectively, are satisfied.
\end{dfn}

Observe that the difference between $T_{\bf h}$ and $T_{{\bf h}, int}$ is only the magnitude in directional variable; $s$ and $R$.
As seen below, the intermediate compactification induce a trivial equivalence between dynamical systems via $T_{\bf h}$ and $T_{{\bf h}, int}$ as well as a nontrivial equivalence between dynamical systems via $T_{qP}$ and $T_{{\bf h}, int}$.

\subsection{Correspondence between compactifications}
Now we have three types of compactifications.
It is desirable that these compactifications are transformed (at least) homeomorphically so that trajectories of dynamical systems in individual compactifications correspond homeomorphically.
In this section, we discuss relationship between the above compactifications.
\par
Let $T_{qP}$ be the quasi-Poincar\'{e} compactification, $T_{\bf h}$ be the directional compactification and $T_{{\bf h}, int}$ be the intermediate compactification associated with $T_{\bf h}$.
Assume that the type $\alpha$ of these compactifications and a sequence of positive numbers $\{a_i\}_{i=1}^n$ are identical.
Then, by definition we have
\begin{equation*}
y_i = \kappa(y)^{\alpha_i}x_i = \frac{h_i(\theta_1,\cdots, \theta_{n-1})}{s^{\alpha_i}},
\end{equation*}
to obtain
\begin{equation*}
p(y)^{2c} = \sum_{i\in I_\alpha} a_i y_i^{2\beta_i} = \kappa^{2c}\sum_{i\in I_\alpha} a_i x_i^{2\beta_i} = s^{-2c}\sum_{i\in I_\alpha} h_i^{2\beta_i}
\end{equation*}
in the domain of $T_{\bf h}^{-1}$.
For $s > 0$, equivalently, for $x$ with $\sum_{i\in I_\alpha} a_i x_i^{2\beta_i} < 1$ (from Definition \ref{dfn-dir-cpt}),
\begin{align*}
&\left(1-\sum_{i\in I_\alpha} a_i x_i^{2\beta_i} \right)\sum_{i\in I_\alpha} h_i^{2\beta_i} = s^{2c}\sum_{i\in I_\alpha} a_i x_i^{2\beta_i}\\
&\Leftrightarrow \left(1-\sum_{i\in I_\alpha} a_i x_i^{2\beta_i} \right) = \frac{s^{2c}}{s^{2c} + \sum_{i\in I_\alpha} a_i h_i^{2\beta_i}}\\
&\Leftrightarrow \kappa = s^{-1}\left( s^{2c} + \sum_{i\in I_\alpha} a_i h_i^{2\beta_i} \right)^{1/2c}.
\end{align*}
Therefore we obtain
\begin{equation}
\label{x-by-h}
x_i = \frac{y_i}{\kappa^{\alpha_i}} = h_i\left( s^{2c} + \sum_{i\in I_\alpha} a_i h_i^{2\beta_i} \right)^{-\alpha_i/2c}.
\end{equation}
Similarly, for $\sum_{i\in I_\alpha} a_i x_i^{2\beta_i} \in (0,1)$, we have
\begin{equation}
\label{h-by-x}
h_i = s^{\alpha_i}\kappa(y)^{\alpha_i} x_i = (e(y)\kappa(y))^{\alpha_i} x_i.
\end{equation}
Since (\ref{x-by-h}) is given by the composite $T_{qP}\circ T_{\bf h}^{-1}$, $T_{qP}$ is diffeomorphic on $\mathbb{R}^n$ and $T_{\bf h}$ is diffeomorphic on $\{s>0\}\times M$, then $T_{qP}\circ T_{\bf h}^{-1}$ is a diffeomorphism from $\{s>0\}\times M$ onto $T_{qP}\circ T_{\bf h}^{-1}(\{s>0\}\times M)$.
\par
By the assumption $\sum_{i\in I_\alpha} a_i h_i^{2\beta_i} \geq c_h > 0$, then the expression (\ref{x-by-h}) can be continuously extended on the horizon $\{s=0\}$.
Since $x=(x_1,\cdots, x_n)$ in (\ref{x-by-h}) is the projection of the intersection point of the curve connecting the origin $(0,\cdots, 0)\in \mathbb{R}^{n+1}$ and
\begin{equation*}
(s,\theta_1,\cdots, \theta_{n-1})\mapsto \left(\frac{h_1(\theta_1,\cdots, \theta_{n-1})}{s^{\alpha_1}}, \cdots, \frac{h_n(\theta_1,\cdots, \theta_{n-1})}{s^{\alpha_n}}, s \right)
\end{equation*}
and the quasi-hemisphere $\{p(x)^{2c} + s^{2c} = 1\}$ on $\mathbb{R}^n$, then it is determined uniquely, which implies that the mapping $T_{qP}\circ T_{\bf h}^{-1}$ is injective on $\{s=0\}$.
\par
Similarly, by the definition of $e$, the term $e(y)\kappa(y)$ is bounded locally uniformly as $p(y)\to \infty$.
Therefore the expression (\ref{h-by-x}) can be continuously extended on the horizon $\{p(x) = 1\}$.

\begin{prop}
\label{prop-diffeo-change-of-coordinates}
Let $\alpha = (\alpha_1,\cdots, \alpha_n)\in \mathbb{Z}^n_{\geq 0}\setminus \{(0,\cdots, 0)\}$ be given.
Then the change of coordinate
\begin{equation}
\label{change-of-coordinate}
C_{{\bf h},int \to qP}: (R,\theta_1,\cdots, \theta_{n-1}) \mapsto (x_1,\cdots, x_n)
\end{equation}
is locally diffeomorphic including $R=0$, where $(x_1,\cdots, x_n) = T_{qP}(y_1,\cdots, y_n)$ is the coordinate determined by the quasi-Poincar\'{e} compactification $T_{qP}$  of type $\alpha$, and 
$(R,\theta_1,\cdots, \theta_{n-1}) = T_{{\bf h},int}(y_1,\cdots, y_n)$ is the coordinate determined by the intermediate compactification $T_{int}$ of type $\alpha$.
\end{prop}
\begin{proof}
First we give the mapping (\ref{change-of-coordinate}) explicitly.
It immediately holds that
\begin{equation*}
y_i = \left(1-\sum_{i\in I_\alpha}a_i x_i^{2\beta_i}\right)^{-\alpha_i/2c} x_i = R^{-\alpha_i/2c} h_i(\theta_1,\cdots, \theta_{n-1}),
\end{equation*}
which yields
\begin{equation*}
\kappa(y)^{2c} = \left(1-\sum_{i\in I_\alpha}a_i x_i^{2\beta_i}\right)^{-1} \sum_{i\in I_\alpha}a_i x_i^{2\beta_i} = R^{-1} \sum_{i\in I_\alpha}a_i h_i(\theta_1,\cdots,\theta_{n-1})^{2\beta_i}.
\end{equation*}
Let $X := \sum_{i\in I_\alpha}a_i x_i^{2\beta_i}$ and $H = \sum_{i\in I_\alpha}a_i h_i(\theta_1,\cdots,\theta_{n-1})^{2\beta_i}$.
The above equality yields
\begin{equation*}
\frac{X}{1-X} = R^{-1} H \quad \Leftrightarrow \quad X = \frac{H}{R+H} \text{ and }1-X = \frac{R}{R+H}.
\end{equation*}
This expression makes sense even for $R=0$, equivalently $X=1$.
Then $C_{{\bf h},int \to qP}$ is represented by
\begin{equation*}
x_i = \frac{h_i(\theta_1,\cdots, \theta_{n-1})}{R+H}.
\end{equation*}
Our claim here is then that the Jacobian matrix of $C_{{\bf h},int \to qP}$ is invertible and continuous.
Assume that $i\in I_{\alpha}$.
Direct computations yield
\begin{align*}
\frac{\partial x_i}{\partial R} &= -\frac{\alpha_i}{2c}(R+H)^{-\frac{\alpha_i}{2c}-1}h_i,\\
\frac{\partial x_i}{\partial \theta_j} &= -\frac{\partial h_i}{\partial \theta_j}(R+H)^{-\frac{\alpha_i}{2c}} -\frac{\alpha_i}{2c}(R+H)^{-\frac{\alpha_i}{2c}-1}h_i \tilde H_j,
\end{align*}
where $\tilde H_j = \sum_{k\in I_\alpha} 2a_k \beta_k h_k^{2\beta_k-1} \frac{\partial h_k}{\partial \theta_j}$.
Note that these expressions still make sense for $i\not \in I_\alpha$; namely, $\alpha_i = 0$.
Thus the Jacobian matrix $JC_{{\bf h},int \to qP}$ at $(R,\theta_1,\cdots, \theta_{n-1})$ is
\begin{align*}
JC_{{\bf h},int \to qP} &= 
\begin{pmatrix}
\frac{\partial x_1}{\partial R} & \frac{\partial x_1}{\partial \theta_1} & \cdots & \frac{\partial x_1}{\partial \theta_{n-1}}\\
\frac{\partial x_2}{\partial R} & \frac{\partial x_2}{\partial \theta_1} & \cdots & \frac{\partial x_2}{\partial \theta_{n-1}}\\
\vdots & \vdots & \vdots & \vdots \\
\frac{\partial x_n}{\partial R} & \frac{\partial x_n}{\partial \theta_1} & \cdots & \frac{\partial x_n}{\partial \theta_{n-1}} 
\end{pmatrix}\\
&=
D \begin{pmatrix}
-\frac{\alpha_1}{2c}(R+H)^{-1}h_1 & -\frac{\partial h_1}{\partial \theta_1} -\frac{\alpha_1}{2c}(R+H)^{-1}h_1 \tilde H_1 & \cdots & -\frac{\partial h_1}{\partial \theta_{n-1}} -\frac{\alpha_1}{2c}(R+H)^{-1}h_1 \tilde H_{n-1}\\
-\frac{\alpha_2}{2c}(R+H)^{-1}h_2 & -\frac{\partial h_2}{\partial \theta_1} -\frac{\alpha_2}{2c}(R+H)^{-1}h_2 \tilde H_1 & \cdots & -\frac{\partial h_2}{\partial \theta_{n-1}} -\frac{\alpha_2}{2c}(R+H)^{-1}h_2 \tilde H_{n-1}\\
\vdots & \vdots & \vdots & \vdots \\
-\frac{\alpha_n}{2c}(R+H)^{-1}h_n & -\frac{\partial h_n}{\partial \theta_1} -\frac{\alpha_n}{2c}(R+H)^{-1}h_n \tilde H_1 & \cdots & -\frac{\partial h_n}{\partial \theta_{n-1}} -\frac{\alpha_n}{2c}(R+H)^{-1}h_n \tilde H_{n-1}\\
\end{pmatrix}\\
&\equiv D\tilde A,
\end{align*}
where 
\begin{equation*}
D = \begin{pmatrix}
(R+H)^{-\alpha_1/2c} & 0 & \cdots & 0\\
0 & (R+H)^{-\alpha_2/2c} & \cdots & 0\\
\vdots & \vdots & \vdots & \vdots\\
0 & 0 & \cdots & (R+H)^{-\alpha_n/2c}
\end{pmatrix}.
\end{equation*}
The diagonal matrix $D$ is nonsingular for $R\geq 0$, since $H \geq c_h > 0$ by (Dir2) in Definition \ref{dfn-directional}.
It is thus sufficient to show that $\tilde A$ is invertible and continuous to our claim. 
The matrix $\tilde A$ is also represented as
\begin{equation*}
A = \tilde A P_{n-1} \cdots P_1 \tilde P, 
\end{equation*}
where $A = A(\theta_1,\cdots, \theta_{n-1})$ is the matrix in (\ref{A-theta}) and 
\begin{equation*}
(P_k)_{ij} = \begin{cases}
1 & \text{if $i=j$}\\
-\tilde H_k & \text{if $i=1$ and $j=k+1$}\\
0 & \text{otherwise}
\end{cases}\quad \text{ and }\quad (\tilde P)_{ij} = -(2c(R+H)\delta_{i1})\delta_{ij}.
\end{equation*}
All matrices $P_1,\cdots, P_{n-1}$ and $\tilde P$ are invertible\footnote
{
This transformation is nothing but the Gaussian elimination for the squared matrix $\tilde A$.
}, and hence the invertibility of $\tilde A$ is equivalent to that of $A$.
By assumption (Dir4) in Definition \ref{dfn-directional}, $A$ is invertible and hence the Jacobian matrix $JC_{{\bf h},int \to qP}$ is invertible for any points on $\{R\geq 0\}\times M$.
As a consequence, thank to Inverse Mapping Theorem, the mapping $C_{{\bf h},int \to qP}$ is locally diffeomorphic.
\end{proof}

\section{Compactifications and dynamics at infinity}
\label{section-dyn-infty}

Now we have the expression of infinity via several type of compactifications.
In this section, we consider the vector field corresponding to (\ref{ODE-original}) in the coordinate of compactified spaces.

\subsection{Desingularized vector fields for quasi-Poincar\'{e} compactifications}
First we calculate the vector field (\ref{ODE-original}) after the quasi-Poincar\'{e} compactification $T_{qP}$.
Integers $\{\beta_i\}_{i=1}^n$ and $c$ in the definition of $T_{qP}$ are assumed to satisfy (\ref{LCM}).
Differentiating $x = T_{qP}(y)$ with respect to $t$, we have 

\begin{align*}
x_i' &= y_i' = f_i(y) \quad \text{ if }i\not \in I_\alpha,\\
x_i' &= \left(\frac{y_i}{\kappa^{\alpha_i}}\right)' = \frac{y_i'}{\kappa^{\alpha_i}} -  \frac{\alpha_i y_i \kappa^{\alpha_i-1}}{\kappa^{2\alpha_i}}\kappa' \\
	&= \frac{y_i'}{\kappa^{\alpha_i}} -  \frac{\alpha_i y_i}{\kappa^{\alpha_i+1}}\left(\sum_{j\in I_\alpha}\frac{\beta_j a_j}{c\kappa^{2c-1}} y_j^{2\beta_j-1}y_j' \right)
	= \frac{f_i(y)}{\kappa^{\alpha_i}} -  \frac{\alpha_i y_i}{\kappa^{\alpha_i+2c}}\left(\sum_{j\in I_\alpha} \frac{a_j}{\alpha_j} y_j^{2\beta_j-1}f_j(y) \right)\quad \text{ if }i\in I_\alpha,
\end{align*}
Namely,
\begin{equation}
\label{ODE-poincare}
x' = A_\alpha \left(f(y) - \kappa^{-1}\langle f, \nabla \kappa \rangle y_\alpha \right)
\end{equation}
with an appropriate permutation of coordinates.
We have the one-to-one correspondence of {\em bounded} equilibria, which helps us with detecting dynamics at infinity.
\begin{prop}
The quasi-Poincar\'{e} compactification $T_{qP}$ maps bounded equilibria of (\ref{ODE-original}) in $\mathbb{R}^n$ into equilibria of (\ref{ODE-poincare}) in $\mathcal{D}$, and vice versa. 
\end{prop}
\begin{proof}
Suppose that $y_\ast$ is an equilibrium of (\ref{ODE-original}), i.e., $f(y_\ast) = 0$.
Then the right-hand side of (\ref{ODE-poincare}) obviously vanishes at the corresponding $x_\ast$.
\par
Conversely, suppose that the right-hand side of (\ref{ODE-poincare}) vanishes at a point $x\in \mathcal{D}, p(x) < 1$: namely,
\begin{equation*}
f(\kappa x) - \kappa(y)^{-1}\langle \nabla \kappa, f(\kappa x) \rangle y_\alpha = 0.
\end{equation*}
Multiplying $\nabla \kappa$, we have
\begin{equation*}
\langle \nabla \kappa, f(\kappa x) \rangle \left(1- \kappa(y)^{-1}\langle \nabla \kappa, y_\alpha \rangle \right) = 0.
\end{equation*}
Due to (A3) in Proposition \ref{prop-adm}-2, we have $|\kappa(y)^{-1}\langle \nabla \kappa, y_\alpha \rangle| < 1$ and hence $\langle \nabla \kappa, f(\kappa x) \rangle = 0$.
Thus we have $f(y) = f(\kappa x) = 0$ by the assumption.
\end{proof}

Next we discuss the dynamics at infinity. Denoting
\begin{equation}
\label{f-tilde}
\tilde f_j(x_1,\cdots, x_n) := \kappa^{-(k+\alpha_j)} f_j(\kappa^{\alpha_1}x_1, \cdots, \kappa^{\alpha_n}x_n),\quad j=1,\cdots, n,
\end{equation}
we have
\begin{align}
\notag
x_i' &= \kappa^{k} \tilde f_i(x)\quad \text{ if }i\not \in I_\alpha,\\
\label{vectorfield-cw}
x_i' &=  \frac{\kappa^{k+\alpha_i} \tilde f_i(x)}{\kappa^{\alpha_i}} - \frac{\kappa^{\alpha_i} x_i}{\beta_i \kappa^{\alpha_i+2c}}\left(\sum_{j\in I_\alpha}\beta_j a_j (\kappa^{\alpha_j} x_j)^{2\beta_j-1}\kappa^{k+\alpha_j} \tilde f_j(x) \right)\quad \text{ if }i\in I_\alpha.
\end{align}
Since $\kappa \to \infty$ as $p \to 1$, then the vector field has singularities at infinity, while $\tilde f_j(x)$ themselves are continuous on $\overline{\mathcal{D}}$.
Nevertheless, the definition of quasi-Poincar\'{e} compactification yields the following observation.
\begin{lem}
\label{lem-order}
The right-hand side of (\ref{vectorfield-cw}) is $O(\kappa^k)$ as $\kappa\to \infty$ no matter whether or not $i\in I_\alpha$.
In other words, the order with respect to $\kappa$ is independent of components of asymptotically quasi-homogeneous vector fields.
\end{lem}

\begin{proof}
The case $i\not \in I_\alpha$ is obvious. 
Consider then the case $i \in I_\alpha$.
\par
By definition $\tilde f_i$ is $O(1)$ as $\kappa\to \infty$, and hence the first term in the right-hand side of (\ref{vectorfield-cw}) is $O(\kappa^k)$.
Check the second term.
Direct calculations yield
\begin{align*}
 \frac{\kappa^{\alpha_i} x_i}{\beta_i \kappa^{\alpha_i+2c}}\left(\sum_{j\in I_\alpha} \beta_j a_j (\kappa^{\alpha_j} x_j)^{2\beta_j-1}\kappa^{k+\alpha_j} \tilde f_j(x) \right) &=  \frac{x_i}{\beta_i \kappa^{2c}}\left(\sum_{j\in I_\alpha} \beta_j a_j \kappa^{\alpha_j(2\beta_j -1)+k+\alpha_j} x_j^{2\beta_j-1} \tilde f_j(x) \right)\\
	&=  \frac{x_i}{\beta_i \kappa^{2c}}\left(\sum_{j\in I_\alpha} \beta_j a_j \kappa^{2\alpha_j \beta_j +k} x_j^{2\beta_j-1} \tilde f_j(x) \right)\\
	&=  \frac{x_i \kappa^k}{\beta_i }\left(\sum_{j\in I_\alpha} \beta_j a_j x_j^{2\beta_j-1} \tilde f_j(x) \right),
\end{align*}
where we used the condition $\alpha_j \beta_j \equiv c$ for all $j$ from (\ref{LCM}).
Since $\tilde f_i$ is $O(1)$ as $\kappa\to \infty$, then the second term in the right-hand side of (\ref{vectorfield-cw}) is $O(\kappa^k)$ as $\kappa\to \infty$.
Summarizing our observations, the right-hand side of (\ref{vectorfield-cw}) is $O(\kappa^k)$ as $\kappa\to \infty$.
\end{proof}

\begin{rem}\rm
In the case of $n=1$, regard the order of quasi-homogeneous functions $f$ as $k+1 \equiv k+\alpha_1$ instead of $k$, which is compatible with the general case containing $I_\alpha = \{i\}$ for some $i\in \{1,\cdots, n\}$.
In such a case, Lemma \ref{lem-order} still holds for $n=1$ without any modifications, unlike the arguments in \cite{EG2006}.
\end{rem}
Lemma \ref{lem-order} leads to introduce the following transformation of time variable.

\begin{dfn}[Time-variable desingularization]\rm 
Define the new time variable $\tau$ depending on $y$ by
\begin{equation}
\label{time-desing}
d\tau = \kappa(y(t))^{k} dt,
\end{equation}
equivalently,
\begin{equation*}
t - t_0 = \int_{\tau_0}^\tau \frac{d\tau}{\kappa(y(\tau))^k},
\end{equation*}
where $\tau_0$ and $t_0$ denote the correspondence of initial times, and $y(\tau)$ is the solution trajectory $y(t)$ under the parameter $\tau$.
We shall call (\ref{time-desing}) {\em the desingularization of (\ref{vectorfield-cw}) of order $k+1$}.
\end{dfn}

The vector field (\ref{vectorfield-cw}) is then desingularized in $\tau$-time scale:
\begin{equation}
\label{ODE-desing}
\dot x_i \equiv \frac{dx_i}{d\tau} = \tilde f_i - \left(\sum_{j\in I_\alpha} \beta_j a_j x_j^{2\beta_j - 1}\tilde f_j \right) \frac{x_i}{\beta_i} =  \tilde f_i - \alpha_i \kappa^{\alpha_i - 1} \langle \nabla \kappa, \tilde f \rangle x_i \equiv g_i(x).
\end{equation}

In particular, we have the extension of dynamics at infinity.
\begin{prop}[Extension of dynamics at infinity]
\label{prop-ext}
Let $\tau$ be the new time variable given by (\ref{time-desing}).
Then the dynamics (\ref{ODE-original}) can be extended to the infinity in the sense that the vector field $g$ in (\ref{ODE-desing}) is continuous on $\overline{\mathcal{D}}$, in particular, on $\mathcal{E} = \{p(x) = 1\}$.
\end{prop}

\begin{proof}
The component-wise desingularized vector field (\ref{ODE-desing}) is obviously continuous on $\overline{\mathcal{D}}$ since this consists of product and sum of continuous functions $x_i$'s and $\tilde f_i$'s on $\overline{\mathcal{D}}$.
\end{proof}

\subsection{Desingularized vector fields for directional compactifications}
Next consider the vector field corresponding to (\ref{ODE-desing}) in the coordinate of directional compactification $T_{\bf h}$.
Following (\ref{dfn-dir-cpt}), we have
\begin{equation*}
y_i' = 
\begin{cases}
f_i(y), & \text{ if }i\not \in I_\alpha,\\
\displaystyle{ -\alpha_i s^{-(\alpha_i +1)} h_i s' + s^{-\alpha_i} \sum_{j=1}^{n-1}\frac{\partial h_i}{\partial \theta_j} \theta_j' }, & \text{ if }i \in I_\alpha,
\end{cases}
\end{equation*}
while its vector- and matrix-form is
\begin{equation}
\label{vec-directional}
\begin{pmatrix}
y_1' \\ y_2' \\ \vdots \\ y_n'
\end{pmatrix}
=
\begin{pmatrix}
\alpha_1 s^{-(\alpha_1+1)} h_1 & s^{-\alpha_1} \frac{\partial h_1}{\partial \theta_1} & \cdots & s^{-\alpha_1} \frac{\partial h_1}{\partial \theta_{n-1}}\\
\alpha_2 s^{-(\alpha_2+1)} h_2 & s^{-\alpha_2} \frac{\partial h_2}{\partial \theta_1} & \cdots & s^{-\alpha_2} \frac{\partial h_2}{\partial \theta_{n-1}}\\
\vdots & \vdots & \ddots & \vdots \\
\alpha_n s^{-(\alpha_n+1)} h_n & s^{-\alpha_n} \frac{\partial h_n}{\partial \theta_1} & \cdots & s^{-\alpha_n} \frac{\partial h_n}{\partial \theta_{n-1}}
\end{pmatrix}
\begin{pmatrix}
s' \\ \theta_1' \\ \vdots \\ \theta_{n-1}'
\end{pmatrix}
\equiv D_s\begin{pmatrix}
s' \\ \theta_1' \\ \vdots \\ \theta_{n-1}'
\end{pmatrix}.
\end{equation}
It easily follows that the matrix $D_s$ is written by the following product of matrices:
\begin{equation*}
D_s = \begin{pmatrix}
s^{-\alpha_1} & 0 & \cdots & 0 \\
0 & s^{-\alpha_2} & \cdots & 0\\
\vdots & \vdots & \ddots & \vdots \\
0 & 0 & \cdots & s^{-\alpha_n}
\end{pmatrix}
\begin{pmatrix}
\alpha_1 h_1 & \frac{\partial h_1}{\partial \theta_1} & \cdots & \frac{\partial h_1}{\partial \theta_{n-1}}\\
\alpha_2 h_2 & \frac{\partial h_2}{\partial \theta_1} & \cdots & \frac{\partial h_2}{\partial \theta_{n-1}}\\
\vdots & \vdots & \ddots & \vdots \\
\alpha_n h_n & \frac{\partial h_n}{\partial \theta_1} & \cdots & \frac{\partial h_n}{\partial \theta_{n-1}}
\end{pmatrix}
\begin{pmatrix}
-s^{-1} & 0 & \cdots & 0 \\
0 & 1 & \cdots & 0\\
\vdots & \vdots & \ddots & \vdots \\
0 & 0 & \cdots & 1
\end{pmatrix}.
\end{equation*}
The middle matrix in the right-hand side is exactly the matrix $A(\theta_1,\cdots, \theta_{n-1})$ in Definition \ref{dfn-dir-cpt} and hence, by assumption, the matrix $A$ is invertible on $M$.
Let $B = B(\theta_1,\cdots, \theta_{n-1})$ be the inverse of $A$.
Consequently, the matrix $D_s$ is invertible on $\{s>0\}\times M$ to obtain
\begin{equation*}
D_s^{-1} = \begin{pmatrix}
-s & 0 & \cdots & 0 \\
0 & 1 & \cdots & 0\\
\vdots & \vdots & \ddots & \vdots \\
0 & 0 & \cdots & 1
\end{pmatrix}
B
\begin{pmatrix}
s^{\alpha_1} & 0 & \cdots & 0 \\
0 & s^{\alpha_2} & \cdots & 0\\
\vdots & \vdots & \ddots & \vdots \\
0 & 0 & \cdots & s^{\alpha_n}
\end{pmatrix}.
\end{equation*}
Therefore (\ref{vec-directional}) in $\{s>0\}\times M$ is equivalent to
\begin{equation}
\label{vec-directional-inv}
\begin{pmatrix}
s' \\ \theta_1' \\ \vdots \\ \theta_{n-1}'
\end{pmatrix}
=
\begin{pmatrix}
-s & 0 & \cdots & 0 \\
0 & 1 & \cdots & 0\\
\vdots & \vdots & \ddots & \vdots \\
0 & 0 & \cdots & 1
\end{pmatrix}
B
\begin{pmatrix}
s^{\alpha_1} & 0 & \cdots & 0 \\
0 & s^{\alpha_2} & \cdots & 0\\
\vdots & \vdots & \ddots & \vdots \\
0 & 0 & \cdots & s^{\alpha_n}
\end{pmatrix}
\begin{pmatrix}
y_1' \\ y_2' \\ \vdots \\ y_n'
\end{pmatrix}.
\end{equation}
Similarly to (\ref{f-tilde}), let
\begin{equation}
\label{f-tilde-directional}
\hat f_j(s, \theta_1, \cdots, \theta_{n-1}) := s^{k+\alpha_j} f_j(s^{-\alpha_1}h_1, \cdots, s^{-\alpha_n}h_n)\text{ with }h_i = h_i(\theta_1,\cdots, \theta_{n-1}),\quad j=1,\cdots, n.
\end{equation}
Then (\ref{vec-directional-inv}) is rewritten as
\begin{equation}
\label{vec-directional-inv-2}
\begin{pmatrix}
s' \\ \theta_1' \\ \vdots \\ \theta_{n-1}'
\end{pmatrix}
=
s^{-k}\begin{pmatrix}
-s & 0 & \cdots & 0 \\
0 & 1 & \cdots & 0\\
\vdots & \vdots & \ddots & \vdots \\
0 & 0 & \cdots & 1
\end{pmatrix}
B
\begin{pmatrix}
\hat f_1 \\ \hat f_2 \\ \vdots \\ \hat f_n
\end{pmatrix}.
\end{equation}
The form of $\hat f_i$ in (\ref{f-tilde-directional}) and asymptotic quasi-homogeneity of $f_i$ and $s$-independence of the matrix $B$ immediately yield the following consequence, which is the directional compactifications' analogue of Lemma \ref{lem-order}.

\begin{lem}
\label{lem-order-directional}
The right-hand side of (\ref{vec-directional-inv-2}) is $O(s^{-k})$ as $s\to 0$ no matter whether or not $i\in I_\alpha$.
More precisely, the $s$-component of (\ref{vec-directional-inv-2}) is $O(s^{-k+1})$ as $s\to 0$.
\end{lem}

Lemma \ref{lem-order-directional} leads to introduce the following transformation of time variable.

\begin{dfn}[Time-variable desingularization (directional compactification version)]\rm 
Define the new time variable $\tau_d$ depending on $y$ by
\begin{equation}
\label{time-desing-directional}
d\tau_d = s(t)^{-k} dt
\end{equation}
equivalently,
\begin{equation*}
t - t_0 = \int_{\tau_0}^\tau s(\tau_d)^k d\tau_d,
\end{equation*}
where $\tau_0$ and $t_0$ denote the correspondence of initial times, and $s(\tau_d)$ is the solution trajectory $s(t)$ under the parameter $\tau$.
We shall call (\ref{time-desing-directional}) {\em the desingularization of (\ref{vec-directional-inv-2}) of order $k+1$}.
\end{dfn}

The vector field (\ref{vec-directional-inv-2}) is then desingularized in $\tau$-time scale:
\begin{equation}
\label{ODE-desing-directional}
\begin{pmatrix}
\frac{ds}{d\tau_d} \\ \frac{d\theta_1}{d\tau_d} \\ \vdots \\ \frac{d\theta_{n-1}}{d\tau_d}
\end{pmatrix}
=
\begin{pmatrix}
-s & 0 & \cdots & 0 \\
0 & 1 & \cdots & 0\\
\vdots & \vdots & \ddots & \vdots \\
0 & 0 & \cdots & 1
\end{pmatrix}
B
\begin{pmatrix}
\hat f_1 \\ \hat f_2 \\ \vdots \\ \hat f_n
\end{pmatrix}\equiv g_d(s,\theta_1,\cdots, \theta_{n-1}).
\end{equation}

In particular, we have the extension of dynamics at infinity.
\begin{prop}[Extension of dynamics at infinity (directional compactification version)]
\label{prop-ext-directional}
Let $\tau_d$ be the new time variable given by (\ref{time-desing-directional}).
Then the dynamics (\ref{ODE-original}) can be extended to the infinity in the sense that the vector field $g_d$ in (\ref{ODE-desing-directional}) is continuous on $\{s\geq 0\}\times M$.
\end{prop}

The assumption $s\sim \kappa^{-1}$ as $s\to 0$ in Definition \ref{dfn-dir-cpt} ensures that the order of time-scales $\tau$ and $\tau_d$ is identical near infinity.

\par
\bigskip
We can derive the desingularized vector field associated with (\ref{ODE-original}) via the intermediate cmpactification $T_{{\bf h},int}$ in the similar way.
Now we have
\begin{equation}
\label{vec-intermediate}
\begin{pmatrix}
y_1' \\ y_2' \\ \vdots \\ y_n'
\end{pmatrix}
=
\begin{pmatrix}
-\frac{\alpha_1}{2c} R^{-(\frac{\alpha_1}{2c}+1)} h_1 & R^{-\frac{\alpha_1}{2c}} \frac{\partial h_1}{\partial \theta_1} & \cdots & R^{-\frac{\alpha_1}{2c}} \frac{\partial h_1}{\partial \theta_{n-1}}\\
-\frac{\alpha_2}{2c} R^{-(\frac{\alpha_2}{2c}+1)} h_2 & R^{-\frac{\alpha_2}{2c}} \frac{\partial h_2}{\partial \theta_1} & \cdots & R^{-\frac{\alpha_2}{2c}}\frac{\partial h_2}{\partial \theta_{n-1}}\\
\vdots & \vdots & \ddots & \vdots \\
-\frac{\alpha_n}{2c} R^{-(\frac{\alpha_n}{2c}+1)} h_n & R^{-\frac{\alpha_n}{2c}} \frac{\partial h_n}{\partial \theta_1} & \cdots & R^{-\frac{\alpha_n}{2c}} \frac{\partial h_n}{\partial \theta_{n-1}}
\end{pmatrix}
\begin{pmatrix}
R' \\ \theta_1' \\ \vdots \\ \theta_{n-1}'
\end{pmatrix}
\equiv D_R\begin{pmatrix}
R' \\ \theta_1' \\ \vdots \\ \theta_{n-1}'
\end{pmatrix}.
\end{equation}
It easily follows that the matrix $D_R$ is written by the following product of matrices:
\begin{equation*}
D_R = \begin{pmatrix}
R^{-\frac{\alpha_1}{2c}} & 0 & \cdots & 0 \\
0 & R^{-\frac{\alpha_2}{2c}} & \cdots & 0\\
\vdots & \vdots & \ddots & \vdots \\
0 & 0 & \cdots & R^{-\frac{\alpha_n}{2c}}
\end{pmatrix}
A
\begin{pmatrix}
-R^{-1} & 0 & \cdots & 0 \\
0 & 1 & \cdots & 0\\
\vdots & \vdots & \ddots & \vdots \\
0 & 0 & \cdots & 1
\end{pmatrix},
\end{equation*}
where $A$ is the matrix $A(\theta_1,\cdots, \theta_{n-1})$ given in (\ref{A-theta}).
Introducing 
\begin{equation}
\label{f-tilde-intermediate}
\bar f_j(R, \theta_1, \cdots, \theta_{n-1}) := R^{(k+\alpha_j)/2c} f_j(R^{-\alpha_1/2c}h_1, \cdots, R^{-\alpha_n/2c}h_n)\text{ with }h_i = h_i(\theta_1,\cdots, \theta_{n-1}),\quad j=1,\cdots, n,
\end{equation}
the vector field (\ref{vec-intermediate}) is rewritten as
\begin{equation}
\label{vec-intermediate-inv}
\begin{pmatrix}
R' \\ \theta_1' \\ \vdots \\ \theta_{n-1}'
\end{pmatrix}
=
R^{-k/2c}\begin{pmatrix}
-R & 0 & \cdots & 0 \\
0 & 1 & \cdots & 0\\
\vdots & \vdots & \ddots & \vdots \\
0 & 0 & \cdots & 1
\end{pmatrix}
B
\begin{pmatrix}
\bar f_1 \\ \bar f_2 \\ \vdots \\ \bar f_n
\end{pmatrix}.
\end{equation}
The form of $\bar f_i$ in (\ref{f-tilde-intermediate}) and asymptotic quasi-homogeneity of $f_i$ and $R$-independence of the matrix $B$ immediately yield the following consequence, which is the intermediate compactifications' analogue of Lemma \ref{lem-order}.

\begin{lem}
\label{lem-order-intermediate}
The right-hand side of (\ref{vec-intermediate-inv}) is $O(R^{-k/2c})$ as $R\to 0$ no matter whether or not $i\in I_\alpha$.
More precisely, the $R$-component of (\ref{vec-intermediate-inv}) is $O(R^{(-k/2c)+1})$ as $R\to 0$.
\end{lem}

Lemma \ref{lem-order-intermediate} leads to introduce the following transformation of time variable.

\begin{dfn}[Time-variable desingularization (intermediate compactification version)]\rm 
Define the new time variable $\tau_{int}$ by
\begin{equation}
\label{time-desing-intermediate}
d\tau_{int} = R(t)^{-k/2c} dt
\end{equation}
equivalently,
\begin{equation*}
t - t_0 = \int_{\tau_0}^\tau R(\tau_{int})^{k/2c} d\tau_{int},
\end{equation*}
where $\tau_0$ and $t_0$ denote the correspondence of initial times, and $R(\tau_{int})$ is the solution trajectory $R(t)$ under the parameter $\tau$.
We shall call (\ref{time-desing-intermediate}) {\em the desingularization of (\ref{vec-intermediate-inv}) of order $k+1$}.
\end{dfn}

The vector field (\ref{vec-intermediate-inv}) is then desingularized in $\tau_{int}$-time scale:
\begin{equation}
\label{ODE-desing-intermediate}
\begin{pmatrix}
\frac{dR}{d\tau_{int}} \\ \frac{d\theta_1}{d\tau_{int}} \\ \vdots \\ \frac{d\theta_{n-1}}{d\tau_{int}}
\end{pmatrix}
=
\begin{pmatrix}
-R & 0 & \cdots & 0 \\
0 & 1 & \cdots & 0\\
\vdots & \vdots & \ddots & \vdots \\
0 & 0 & \cdots & 1
\end{pmatrix}
B
\begin{pmatrix}
\bar f_1 \\ \bar f_2 \\ \vdots \\ \bar f_n
\end{pmatrix}\equiv g_{int}(R,\theta_1,\cdots, \theta_{n-1}).
\end{equation}

In particular, we have the extension of dynamics at infinity.
\begin{prop}[Extension of dynamics at infinity (intermediate compactification version)]
\label{prop-ext-intermediate}
Let $\tau_{int}$ be the new time variable given by (\ref{time-desing-intermediate}).
Then the dynamics (\ref{ODE-original}) can be extended to the infinity in the sense that the vector field $g_{int}$ in (\ref{ODE-desing-intermediate}) is continuous on $\{R\geq 0\}\times M$.
\end{prop}

\subsection{Equivalence of desingularized vector fields at infinity}
Now we have three kinds of desingularized vector fields: (\ref{ODE-desing}), (\ref{ODE-desing-directional}) and (\ref{ODE-desing-intermediate}).
Here we consider the equivalence of these vector fields, which will show that qualitative properties of dynamics at infinity are independent of the choice of compactifications.
\par
Now 
\begin{equation*}
\frac{dx}{dt} = A_\alpha (I_n - \kappa^{-1}y_\alpha (\nabla \kappa)^T)f(y) = \kappa^k g(x),
\end{equation*}
\begin{equation*}
\frac{d(s,\theta)}{dt} = \begin{pmatrix}
-s & 0 & \cdots & 0 \\
0 & 1 & \cdots & 0\\
\vdots & \vdots & \ddots & \vdots \\
0 & 0 & \cdots & 1
\end{pmatrix}
B
\begin{pmatrix}
s^{\alpha_1} & 0 & \cdots & 0 \\
0 & s^{\alpha_2} & \cdots & 0\\
\vdots & \vdots & \ddots & \vdots \\
0 & 0 & \cdots & s^{\alpha_n}
\end{pmatrix}
f(y) = s^{-k}g_d(s,\theta)
\end{equation*}
and
\begin{equation*}
\frac{d(R,\theta)}{dt} = \begin{pmatrix}
-R & 0 & \cdots & 0 \\
0 & 1 & \cdots & 0\\
\vdots & \vdots & \ddots & \vdots \\
0 & 0 & \cdots & 1
\end{pmatrix}
B
\begin{pmatrix}
R^{\alpha_1/2c} & 0 & \cdots & 0 \\
0 & R^{\alpha_2/2c} & \cdots & 0\\
\vdots & \vdots & \ddots & \vdots \\
0 & 0 & \cdots & R^{\alpha_n}
\end{pmatrix}
f(y) = R^{-k/2c}g_{int}(R,\theta).
\end{equation*}
The direct calculations indicate that (\ref{ODE-desing}) and (\ref{ODE-desing-directional}) are {\em not} smoothly equivalent since the change of coordinate $C_{{\bf h}\to qP} : (s,\theta_1,\cdots,\theta_{n-1}) \mapsto (x_1,\cdots, x_n)$ is homeomorphic but not diffeomorphic at $\{s=0\}$.
\par 
In what follows, we apply the intermediate compactification $T_{{\bf h},int}$ associated with $T_{\bf h}$ to showing the equivalence of vector fields at infinity.
Since the transformation $C_{{\bf h},int \to qP}$ is locally diffeomorphic by Proposition \ref{prop-diffeo-change-of-coordinates}, then we have
\begin{equation*}
\frac{dx}{dt} = (JC_{{\bf h},int \to qP}) \frac{d(R,\theta)}{dt}.
\end{equation*}
The corresponding desingularized vector fields are 
\begin{equation*}
\frac{dx}{d\tau} = \kappa^{-k}(y(t))\frac{dx}{dt}\quad \text{ and } \quad\frac{d(R,\theta)}{d\tau_{int}} = R^{k/2c}\frac{d(R,\theta)}{dt},
\end{equation*}
respectively, and hence we have
\begin{equation*}
\frac{dx}{d\tau} = (\kappa(y(t))R^{1/2c})^{-k}(JC_{{\bf h},int \to qP}) \frac{d(R,\theta)}{d\tau_{int}}.
\end{equation*}
By the property $R\sim \kappa^{-2c}$ as $p(y)\to \infty$, as stated in (Dir3) in Definition \ref{dfn-directional}, the factor $(\kappa(y(t))R^{1/2c})^{-k}$ is always positive.
We then have the following statement.
\begin{prop}
Let $x = T_{qP}(y)$ be the coordinate in quasi-Poincar\'{e} compactifications and $(R,\theta) = T_{{\bf h},int}(y)$ be the coordinate in the intermediate compactifications associated with directional ones $T_{\bf h}$ of the same type $\alpha$.
Then the desingularized vector fields $\frac{dx}{d\tau} = g(x)$ and $\frac{d(R,\theta)}{d\tau_{int}} = g_{int}(R,\theta)$ are topological equivalent on $\{R\geq 0\}\times M$.
\end{prop}

Next compare the directional compactification $T_{\bf h}$ and the associated intermediate compactification $T_{{\bf h},int}$.
The change of coordinate $C_{{\bf h}\to {\bf h},int}$ is obviously given by
\begin{equation*}
(R,\theta_1,\cdots, \theta_{n-1}) = C_{{\bf h}\to {\bf h},int}(s,\theta_1,\cdots, \theta_{n-1}) = (s^{2c},\theta_1,\cdots, \theta_{n-1})
\end{equation*}
and
\begin{equation*}
JC_{{\bf h}\to {\bf h},int}(s,\theta_1,\cdots, \theta_{n-1}) = {\rm diag}(2cs^{2c-1},1,\cdots, 1).
\end{equation*}
and hence 
\begin{align*}
\frac{dR}{dt} = \frac{d(s^{2c})}{dt} = 2cs^{2c-1} (\tilde g_d(s,\theta)) s = 2c(\tilde g_d(R,\theta)) R.
\end{align*}
This relationship as well as the monotonicity of function $R = s^{2c}$ in $\{s\geq 0\}$ shows that all trajectories in the $(s,\theta)$-coordinate are mapped one-to-one onto the corresponding ones in the $(R,\theta)$-coordinates.
Since $C_{{\bf h}\to {\bf h},int}$ is homeomorphic in $\{s\geq 0\}\times M$, then the vector field $\frac{d(R,\theta)}{d\tau_{int}} = g_{int}(R,\theta)$ and $\frac{d(s,\theta)}{d\tau_d} = g_d(s,\theta)$ are topologically equivalent.
\par
Summarizing the above arguments, we have the following result, which implies that the dynamics at infinity is independent of the choice of (quasi-Poincar\'{e} and directional) compactifications.
\begin{thm}
\label{thm-equivalence}
Let $x = T_{qP}(y)$ be the coordinate in quasi-Poincar\'{e} compactifications and $(s,\theta) = T_{\bf h}(y)$ be the coordinate in directional compactifications of the same type $\alpha$.
Then the desingularized vector fields $\frac{dx}{d\tau} = g(x)$ and $\frac{d(s,\theta)}{d\tau_{int}} = g_{int}(s,\theta)$ are topological equivalent on $\{s\geq 0\}\times M$.
\end{thm}

\subsection{Dynamics at infinity}
\label{section-dynamics-at-infinity}
We have shown that desingularized vector fields associated with (\ref{ODE-original}) can be defined including infinity via quasi-Poincar\'{e}, directional and intermediate compactifications and that the qualitative properties of dynamics for these vector fields are independent of the choice of compactifications in the sense of topological equivalence.
In this section, we discuss dynamics at infinity and correspondence to divergent solutions of (\ref{ODE-original}).
Here we state a series of notions and results only for quasi-Poincar\'{e} compactifications, as comparison with \cite{EG2006}.
Obvious modifications for other compactifications yield the corresponding results.
\par
\bigskip
For quasi-Poincar\'{e} compactifications, Proposition \ref{prop-ext} show that dynamics and invariant sets at infinity make sense under time-variable desingularizations.
For example, \lq\lq {\em equilibria at infinity}" defined below are well-defined.

\begin{dfn}[Equilibria at infinity]\rm
We say that the vector field (\ref{ODE-original}) has an {\em equilibrium at infinity} in the direction $x_\ast$ if $x_\ast$ is an equilibrium of (\ref{ODE-desing}) on the horizon $\mathcal{E}$.
\end{dfn}

By using equilibria at infinity, blow-up and grow-up solutions (i.e., divergent solutions with $t_{\max}=\infty$) are described in terms of asymptotic behavior for desingularized vector fields.
\begin{thm}[Divergent solutions and asymptotic behavior]
\label{thm-diverge}
Let $y(t)$ be a solution of (\ref{ODE-original}) with the interval of maximal existence time $(a,b)$, possibly $a=-\infty$ and $b=+\infty$.
Assume that $y$ tends to infinity in the direction $x_\ast$ as $t\to b-0$ or $t\to a+0$.
Then $x_\ast$ is an equilibrium of (\ref{ODE-desing}) on $\mathcal{E}$.
\end{thm}

\begin{proof}
The property $b = \sup \{t\mid y(t) \text{ is a solution of (\ref{ODE-original})}\}$ corresponds to the property that
\begin{equation*}
\sup \{\tau\mid x(\tau) = T(y(t)) \text{ is a solution of (\ref{ODE-desing}) in the time variable $\tau$}\} = \infty.
\end{equation*}
Indeed, if not, then $\tau \to \tau_0 < \infty$ and $\lim_{\tau \to \tau_0-0}x(\tau) = x_\ast$ as $t\to b-0$.
The condition $x(\tau) = x_\ast$ is the regular initial condition of (\ref{ODE-desing}).
The vector field (\ref{ODE-desing}) with the new initial point $x(\tau) = x_\ast$ thus has a locally unique solution $x(\tau)$ in a neighborhood of $\tau_0$, which contradicts the maximality of $b$.
Therefore we know that $\tau \to +\infty$ as $t\to b-0$.
Since $\lim_{\tau \to \infty}x(\tau) = x_\ast$, then $x_\ast$ is an equilibrium of (\ref{ODE-desing}) on $\mathcal{E}$.
The similar arguments show that $t\to a+0$ corresponds to $\tau \to -\infty$ and that the same consequence holds true.
\end{proof}

This theorem gives a description of divergent solutions from the viewpoint of dynamical systems; namely, assuming the $C^1$-smoothness of desingularized vector fields (\ref{ODE-desing}) on $\overline{\mathcal{D}}$, divergent solutions in the direction $x_\ast$ correspond to trajectories of (\ref{ODE-desing}) on the stable manifold $W^s(x_\ast)$ of the equilibrium $x_\ast$.
This correspondence opens the door to applications of various results in dynamical systems to divergent solutions.
Before moving to the next section, we gather several properties of dynamics at infinity, which will be useful to concrete studies.

\begin{thm}[Dynamics at infinity]
\label{thm-dyn-infty}
\begin{enumerate}
\item The horizon $\mathcal{E}$ is an invariant manifold of (\ref{ODE-desing}).
\item Dynamics of (\ref{ODE-desing}) on $\mathcal{E}$ are dominated by the following vector field:
\begin{align*}
\dot x_i &= (\tilde f_{\alpha,k})_i - \left(\sum_{j\in I_\alpha} \beta_j a_j x^{2\beta_j - 1}(\tilde f_{\alpha,k})_j \right) \frac{x_i}{\beta_i},\quad i=1,\cdots, n.
\end{align*}
\item Time evolution of $1-p(x)^{2c}$ in $\tau$-time scale is dominated by
\begin{equation*}
\frac{d}{d\tau} (1-p(x)^{2c}) =  -\left(\sum_{j\in I_\alpha} \beta_j a_j x_j^{2\beta_j-1}\tilde f_j\right) (1-p(x)^{2c}).
\end{equation*}
\item Assume that the vector field $f$ in (\ref{ODE-original}) is quasi-homogeneous of type $(\alpha_1,\cdots, \alpha_n)$ and order $k+1$. Then the desingularized vector field $g$ given in (\ref{ODE-desing}) satisfies
\begin{equation}
\label{symmetry}
g_i((-1)^{\alpha_1} x_1,\cdots, (-1)^{\alpha_n} x_n) = (-1)^{k+\alpha_i} g_i(x_1,\cdots, x_n).
\end{equation}
In particular, for any asymptotically quasi-homogeneous vector field $f$ in (\ref{ODE-original}) of type $(\alpha_1,\cdots, \alpha_n)$ and order $k+1$, the desingularized vector field $g$ satisfies (\ref{symmetry}) on $\mathcal{E}$.
\end{enumerate}
\end{thm}

\begin{proof}
1. We prove that $\frac{d}{d\tau}p(x)^{2c} = 0$ on $\mathcal{E} = \{x\in \mathbb{R}^n\mid p(x) = 1\}$.
Direct calculations yield
\begin{align*}
\frac{1}{2}\frac{d}{d\tau}p(x)^{2c} &= \frac{1}{2}\frac{d}{d\tau}\left(\sum_{j\in I_\alpha} a_j x_j^{2\beta_j}\right) = \sum_{j\in I_\alpha} \beta_j a_jx_j^{2\beta_j-1}\frac{dx_j}{d\tau} \\
	&= \sum_{j\in I_\alpha} \beta_j a_j x_j^{2\beta_j-1}\left\{\tilde f_j - \left(\sum_{j\in I_\alpha} \beta_i a_i x_i^{2\beta_i-1}\tilde f_i\right)\frac{x_j}{\beta_j}\right\} \\
	&= \sum_{j\in I_\alpha} \beta_j a_j x_j^{2\beta_j-1}\tilde f_j  - \sum_{j\in I_\alpha} a_j x_j^{2\beta_j} \left(\sum_{j\in I_\alpha} \beta_i a_i x_i^{2\beta_i-1}\tilde f_i\right) \\
	&= \sum_{j\in I_\alpha} \beta_j a_j x_j^{2\beta_j-1}\tilde f_j  - \left(\sum_{j\in I_\alpha} \beta_i a_i x_i^{2\beta_i-1}\tilde f_i\right) = 0\quad \text{ since }p(x) = 1.
\end{align*}
\par
\bigskip
2. It immediately follows from
$\lim_{p(x)\to 1}\tilde f(x) = \tilde f_{\alpha, k}$
by the asymptotic quasi-homogeneity of $f$ and (A1) in Proposition \ref{prop-adm}.
\par
\bigskip
3. It immediately follows from calculations in the proof of statement 1.
\par
\bigskip
4. First observe that $\kappa((-1)^{\alpha_1} x_1,\cdots, (-1)^{\alpha_n} x_n) = \kappa(x_1,\cdots, x_n)$.
Second, for all $i$, the function $\tilde f_i$ is quasi-homogeneous of type $(\alpha_1,\cdots, \alpha_n)$ and order $k+\alpha_i$ by assumption.
Then, for each $i$, we have
\begin{align*}
g_i((-1)^{\alpha_1} x_1,\cdots, (-1)^{\alpha_n} x_n) &= \tilde f_i((-1)^{\alpha_1} x_1,\cdots, (-1)^{\alpha_n} x_n)\\
	& - \left(\sum_{j\in I_\alpha} \beta_j a_j ((-1)^{\alpha_j} x_j)^{2\beta_j - 1}\tilde f_j((-1)^{\alpha_1} x_1,\cdots, (-1)^{\alpha_n} x_n) \right) \frac{(-1)^{\alpha_i} x_i}{\beta_i}\\
	&= (-1)^{k+\alpha_i}\tilde f_i - (-1)^{\alpha_i} \left(\sum_{j\in I_\alpha} (-1)^{\alpha_j} \beta_j a_j x_j^{2\beta_j - 1}(-1)^{k+\alpha_j} \tilde f_j \right) \frac{x_i}{\beta_i}\\
	&= (-1)^{k+\alpha_i}\tilde f_i - (-1)^{k+\alpha_i} \left(\sum_{j\in I_\alpha} \beta_j a_j x_j^{2\beta_j - 1} \tilde f_j \right) \frac{x_i}{\beta_i}\\
	&= (-1)^{k+\alpha_i}g_i(x_1,\cdots, x_n)
\end{align*}
and complete the proof.
\end{proof}

\begin{rem}\rm
Theorem \ref{thm-dyn-infty}-4 shows that the vector field at infinity is equivariant with respect to the symmetry $\iota_\alpha (x)$ defined as
\begin{equation*}
(x_1,\cdots, x_n) \mapsto \iota_\alpha(x) \equiv ((-1)^{\alpha_1} x_1, \cdots, (-1)^{\alpha_n} x_n).
\end{equation*}
on $\mathcal{E}$. 
In particular, if $x\in \mathcal{E}$ is an equilibrium of (\ref{ODE-desing}), then so is $\iota_\alpha(x)$.
In the homogeneous case, the symmetry is just $\iota_\alpha(x) = -x$, as stated in Proposition 2.6 of \cite{EG2006}. 
\end{rem}

\section{Blow-up solutions and their asymptotic behavior}
\label{section-blow-up}

Theorem \ref{thm-diverge} indicates that trajectories for (\ref{ODE-desing}) tending to equilibria at infinity as $\tau\to \infty$ are divergent solutions of original system (\ref{ODE-original}).
On the other hand, Theorem \ref{thm-diverge} itself does not distinguish blow-up solutions from grow-up solutions.
Under additional assumptions to equilibria at infinity, we can characterize blow-up solutions from the viewpoint of dynamical systems.
In this section, we give criteria of blow-ups which are sufficient to apply in the following successive sections.
\par
As Section \ref{section-dynamics-at-infinity}, we only show results for quasi-Poincar\'{e} compactifications.
Note that all the following arguments are independent of the choice of coordinates and, thanks to topological equivalence; Theorem \ref{thm-equivalence}, all statements are also valid for directional and intermediate compactifications with suitable modifications.

\subsection{Stationary blow-up}
\label{section-stationary}
Blow-up criterion with homogeneous compactification is discussed in \cite{EG2006}.
Roughly speaking, preceding results stated there claim that {\em linearly stable equilibria at infinity induce blow-up solutions}.
In general, however, equilibria at infinity may admit unstable directions; namely, the Jacobian matrix $Jg$ of (\ref{ODE-desing}) at those points may admit eigenvalues with positive real parts.
Global trajectories asymptotic to such equilibria at infinity will be referred to as {\em unstable grow-up or blow-up solutions}.
The following theorem is one of our main results, which gives criteria of blow-ups not only for stable but also unstable blow-up solutions.

\begin{thm}[Stationary blow-up]
\label{thm-stationary-blowup}
Assume that (\ref{ODE-original}) has an equilibrium at infinity in the direction $x_\ast$.
Suppose that the desingularized vector field $g$ in (\ref{ODE-desing}) is $C^1$ on an open set $V\subset \overline{\mathcal{D}}$ with $\mathcal{E}\subset V$, and that $x_\ast$ is hyperbolic with $n_s > 0$ (resp. $n_u = n-n_s$) eigenvalues of $Jg(x_\ast)$ with negative (resp. positive) real parts.
Then, if the solution $y(t)$ of (\ref{ODE-original}) whose image $x = T_{qP}(y)$ is on $W^s(x_\ast)$ for $g$, $t_{\max} < \infty$ holds; namely, $y(t)$ is a blow-up solution.
Moreover,
\begin{equation*}
p(y(t)) \sim c(t_{\max} - t)^{-1/k}\quad \text{ as }\quad t\to t_{\max},
\end{equation*}
where $k+1$ is the order of asymptotically quasi-homogeneous vector field $f$.
Finally, if the $i$-th component $(x_\ast)_i$ of $x_\ast$ with $i\in I_\alpha$ is not zero, then we also have
\begin{equation*}
y_i(t) \sim c(t_{\max} - t)^{-\alpha_i /k}\quad \text{ as }\quad t\to t_{\max}.
\end{equation*}
\end{thm}

\begin{proof}
First note that $g$ as well as the generated flow is assumed to be $C^1$ on $V$, which indicates that the vector fields has an extension into a neighborhood of $V$ in $\mathbb{R}^n$.
We can choose an open neighborhood $\tilde V$ of $x_\ast$ in $\mathbb{R}^n$ and a $C^1$-change of coordinate $h : (z^u, z^s)\mapsto x$ in $\tilde V$ such that $h(0) = x_\ast\in \mathbb{R}^n$ and that $(\ref{ODE-desing})$ in $V$ is mapped into
\begin{align}
\label{eq-s-linearized}
&\dot z^u = \Lambda^u z^u,\quad \dot z^s = \Lambda^s z^s,\\
\notag
&\Lambda^u = {\rm diag}(J(\lambda^u_1; m^u_1), \cdots J(\lambda^u_{k_u}; m^u_{k_u})),\quad \Lambda^s = {\rm diag}(J(\lambda^s_1; m^s_1), \cdots J(\lambda^s_{k_s}; m^s_{k_s})),
\end{align}
by Hartman-Grobman's Theorem, where $\{\lambda_i^u\}_{i=1}^{k_u}$ and $\{\lambda_i^s\}_{i=1}^{k_s}$ are distinct eigenvalues of $Jg(x_\ast)$ with positive and negative real parts, respectively.
$J(\lambda; k)$ denotes the $k$-dimensional Jordan block matrix of $\lambda$, and $\{m^u_k\}_{k=1}^{k_u}$ and $\{m^s_k\}_{k=1}^{k_s}$ denote the dimension of Jordan block matrices associated with $\{\lambda^u_k\}$ and $\{\lambda^s_k\}$, respectively.
Obviously $\sum_{k=1}^{k_u} m^u_k = n_u$ and $\sum_{k=1}^{k_s} m^s_k = n_s$ are required.
\par
Note that, if $\{\lambda_i^u\}_{i=1}^{n_u}$ and/or $\{\lambda_i^u\}_{i=1}^{n_u}$ contain complex conjugate eigenvalues, say $\lambda_i^u$ and $\lambda_{i+1}^u = \overline{\lambda_i^u}$, then the corresponding diagonal part ${\rm diag}(\lambda_i^u, \lambda_{i+1}^u)$ of $\Lambda^u$ is replaced by
\begin{equation*}
\begin{pmatrix}
{\rm Re}\lambda_i^u & -{\rm Im}\lambda_i^u\\
{\rm Im}\lambda_i^u & {\rm Re}\lambda_i^u
\end{pmatrix}.
\end{equation*}
A similar replacement is operated to $\Lambda^s$.
All arguments below do not change under these replacements.
\par
Our focus here is the stable manifold $W^s(x_\ast)$ of $x_\ast$, which is transformed via the conjugacy $h$ into $\{z^u=0\}\cap h(V)\subset h(\tilde V)$.
Solutions on the stable manifold $W^s(x_\ast) = W^s(h(0))$ are thus written by
\begin{equation*}
z^u = 0,\quad z^s(\tau) = e^{\Lambda^s \tau} z^s_0,
\end{equation*}
in the $z$-coordinate, where $(0, z^s_0)$ is an initial position of solution, which may be assumed to be in $h(\tilde V)$.
Consequently, we have
\begin{equation*}
z(\tau) = c_1 \tau^{m^s_1-1} e^{{\rm Re}\lambda^s_1\tau}({\bf 1}+o({\bf 1}))\quad \text{ as }\tau \to \infty
\end{equation*}
with some constant $c_1$, where ${\bf 1} = (1,1,\cdots, 1)^T\in \mathbb{R}^n$\footnote
{
In this estimate, the leading stable eigenvalue $\lambda^s_1$ and its multiplicity is essential.
Indeed, if $\lambda^s_1$ is simple and $\lambda^s_2$ is double (with geometric multiplier $1$), for example, then the asymptotic behavior of corresponding eigendirections has the order $O(e^{-\lambda^s_1 \tau})$ and $O(\tau e^{-\lambda^s_2 \tau})$, respectively.
However, we immediately have $\tau e^{-(\lambda^s_2 - \lambda^s_1) \tau} \to 0$ as $\tau \to \infty$, since $\lambda^s_2 > \lambda^s_1$, which indicates that $\tau e^{-\lambda^s_2 \tau} = o(e^{-\lambda^s_1 \tau})$.
The asymptotic behavior of trajectories is thus dominated by decays associated with $\lambda^s_1$.
}. 
Hence, thanks to the conjugacy $h = id + v$ with bounded continuous function $v$,
\begin{equation}
\label{asymptotic-x}
x(\tau) - x_\ast = h(z(\tau)) = c_2 \tau^{m^s_1-1} e^{-m\tau}({\bf 1}+o({\bf 1}))\quad \text{ as }\tau \to \infty
\end{equation}
with $-m \leq {\rm Re}\lambda^s_1 < 0$ and some constant $c_2$.
\par
\bigskip
Turn to the quantity $1-p^{2c}(x(\tau))$ for .
Near $x_\ast$ with $p(x_\ast) = 1$,
\begin{align}
\notag
1-p^{2c}(x(\tau)) &= 1-\sum_{i\in I_\alpha} a_i (x_i- (x_\ast)_i + (x_\ast)_i)^{2\beta_i} = 1-\sum_{i\in I_\alpha} \sum_{k_i = 1}^{2\beta_i} a_i {}_{2\beta_i}C_{k_i} (x_\ast)_i^{2\beta_i - k_i}(x_i- (x_\ast)_i)^{k_i}\\
\label{asym-radius}
	&= c_3 \tau^{m^s_1-1} e^{-\ell \tau}(1+o(1))
\end{align}
for some $-\ell \leq -m < 0$, where the terms $1$ and $\sum_{i\in I_\alpha} a_i(x_\ast)_i^{2\beta_i}$ are cancelled out.
Recall that we have for a certain initial point $t_0$,
\begin{align*}
t-t_0 &= \int_0^\tau \frac{d\eta}{\kappa(y(\eta))^{k}}\\
	&= \int_0^\infty \frac{d\eta}{\kappa(y(\eta))^{k}} - \int_\tau^\infty \frac{d\eta}{\kappa(y(\eta))^{k}}
\end{align*}
and the integrals converge due to (\ref{asym-radius})\footnote
{
This estimate holds only for trajectories corresponding to those on the stable manifold of $x_\ast$.
}.
In particular, $t_{\max} < \infty$ holds and the solution is a blow-up solution.
Thus
\begin{equation*}
t_{\max} - t = \int_\tau^\infty \frac{d\eta}{\kappa(y(\eta))^{k}} = c_4 \tau^{k(m^s_1-1)/2c} e^{- (k\ell / 2c) \tau}(1+o(1))\quad \text{ as } \tau \to \infty.
\end{equation*}
Since $dt/d\tau > 0$ on trajectories on $W^s(x_\ast)$ for (\ref{ODE-desing}), this relation is then solvable for $\tau$ and, together with (\ref{asym-radius}), it yields that
\begin{equation*}
1-p^{2c}(x(\tau)) \sim (t_{\max}-t)^{2c/k}
\end{equation*}
and
\begin{equation*}
p(y(t)) = p(x(\tau))\kappa(x(\tau)) = \frac{p(x(\tau))}{(1-p(x(\tau))^{2c})^{1/2c}} \sim c(t_{\max} -t)^{-1/k}
\end{equation*}
as $t\to t_{\max}$, where $c_3, c_4$ and $c$ are certain constants.
\par
It immediately follows from the above asymptotics that, for $y_i$ tending to $x_\ast$ with $(x_\ast)_i \not = 0$,
\begin{equation*}
y_i(t) = \kappa(y(t))^{\alpha_i} x_i(\tau) \sim p(y(t))^{\alpha_i} x_i(\tau) \sim c(t_{\max} -t)^{-\alpha_i/k}\quad \text{ as }t\to t_{\max} .
\end{equation*}
\end{proof}

Theorem \ref{thm-stationary-blowup} generalizes the result in \cite{EG2006} in the sense that the blow-up criteria are valid even for quasi-Poincar\'{e} compactifications, and that the blow-up behavior is characterized by not only stable equilibria on $\mathcal{E}$, but also hyperbolic ones.
We then have the slogan: {\em hyperbolic equilibria at infinity induce blow-up solutions} under the $C^1$-smoothness of desingularized vector fields on $\overline{\mathcal{D}}$.
\par
Remark that the {\em non-resonance condition} of eigenvalues, which is assumed in \cite{EG2006}, is not actually necessary.
\par
Note that the blow-up rate of each component reflects the type $\alpha$ of asymptotically quasi-homogeneous vector field $f$, unlike homogeneous compactifications.
In Sections \ref{section-sing-shock} and \ref{section-two-fluid}, we observe various blow-up solutions in concrete systems, some of which involve equilibria at infinity of saddle type.
Such \lq\lq unstable" blow-up solutions are expected to be the trigger of other singular nature in systems like {\em singular shock profiles} (e.g., \cite{KK1990, SSS1993, S2004}).
\par

\begin{rem}[Lack of smoothness]\rm
\label{rem-extension}
In Theorem \ref{thm-stationary-blowup}, we assumed the $C^1$ smoothness of the desingularized vector field $g$ on $\overline{\mathcal{D}}$.
In fact, in the case of quasi-Poincar\'{e} compactifications, the vector field $g$ may lose its smoothness on $\mathcal{E}$, even if $f$ is arbitrarily smooth.
It is because of the presence of radicals in (quasi-)Poincar\'{e} compactifications.
For example, if $f$ is a polynomial vector field $f(y) = (f_1(y), \cdots, f_n(y))$ with 
\begin{equation}
\label{poly-vf}
f_j(y) = \sum_{i=0}^{k_j} \sum_{i_1,\cdots, i_n\geq 0, i_1+ \cdots +i_n=i} c_{ji_1\cdots i_n}y_i^{i_1}\cdots y_n^{i_n},\quad c_{ji_1\cdots i_n}\not = 0\text{ for some }i_1,\cdots, i_n \text{ with }\sum_{l=1}^n i_l = k_j,
\end{equation}
the loss of smoothness of $g$ may occur. 
Indeed, the form of $\tilde f$ indicate that the desingularized vector field $g_i$ may contain $\kappa^{-\gamma}$ for some $\gamma \in \mathbb{N}$.
A direct calculation yields
\begin{equation*}
\frac{\partial \kappa^{-\gamma}}{\partial x_j} = \frac{\partial }{\partial x_j}\left(1-\sum_{i=1}^n a_i x_i^{2\beta_i}\right)^{\gamma/2c} = -\frac{2\beta_j a_j \gamma}{2c}\left(1-\sum_{i=1}^n a_i x_i^{2\beta_i}\right)^{(\gamma-2c)/2c} x_j^{2\beta_j-1}\quad \text{ with } j\in I_\alpha,
\end{equation*}
which is singular on $\mathcal{E}$ if $\gamma - 2c < 0$.
In the case of polynomial vector fields, we have a rough sufficient condition for $C^1$-extension of $g$ on $\overline{\mathcal{D}}$, which is stated in Lemma \ref{lem-extension}.
\end{rem}

\begin{lem}
\label{lem-extension}
Let $f=(f_1,\cdots, f_n)$ be an asymptotically quasi-homogeneous polynomial vector field of type $\alpha$ and order $k+1$ given in (\ref{poly-vf}).
Suppose that
\begin{equation}
\label{ass-extension}
c_{ji_1\cdots i_n} = 0\text{ for all }i_1,\cdots, i_n \text{ with }\sum_{l=1}^n \alpha_{i_l} \in \{k+\alpha_j -\gamma \mid \gamma = 1,\cdots, 2c-1\}
\end{equation}
holds for all $j =1, \cdots, n$. Then the desingularized vector field $g$ in (\ref{ODE-desing}) is $C^1$ on $\overline{\mathcal{D}}$.
\end{lem}

\begin{proof}
Recall that $g_i = \tilde f_i - \left(\sum_{j=1}^n \beta_j a_j x_j^{2\beta_j - 1}\tilde f_j \right) \frac{x_i}{\beta_i}$.
The vector field $g_i$ is $C^1$ on $\overline{\mathcal{D}}$ if, at least, all $\tilde f_j$ are $C^1$ on $\overline{\mathcal{D}}$.
Observe that
\begin{align*}
\tilde f_j(x_1,\cdots, x_n) &= \kappa^{-(k+\alpha_j)} f_j(\alpha_1 x_1,\cdots, \alpha_n x_n)\\
	&= \kappa^{-(k+\alpha_j)} \sum_{i=0}^{k_j} \sum_{i_1,\cdots, i_n\geq 0, i_1+ \cdots +i_n=i} c_{ji_1\cdots i_n} \kappa^{\alpha_{i_1}+\cdots + \alpha_{i_n}} x_i^{i_1}\cdots x_n^{i_n}.
\end{align*}
From the asymptotical quasi-homogeneity of $f$, $\sum_{l=1}^n \alpha_{i_l} \leq k+\alpha_j$.
Therefore the $\kappa$-term in each summand has the form
\begin{equation*}
\kappa^{\sum_{l=1}^n \alpha_{i_l} - (k+\alpha_j)} = \left(1-\sum_{i=1}^n a_i x_i^{2\beta_i}\right)^{(\sum_{l=1}^n \alpha_{i_l} - (k+\alpha_j))/2c}
\end{equation*}
Arguments in Remark \ref{rem-extension} show that the above term is $C^1$ on $\overline{\mathcal{D}}$ if either $\sum_{l=1}^n \alpha_{i_l} = k+\alpha_j$ or $(k+\alpha_j) - \sum_{l=1}^n \alpha_{i_l} \geq 2c$ holds for all $j$.
Consequently, $\tilde f_j$ is $C^1$ on $\overline{\mathcal{D}}$ if (\ref{ass-extension}) holds for all $j$, and so is $g$, which completes the proof.
\end{proof}

As for directional compactifications, the desingularized vector field $g_d$ is expected to be smooth including $\mathcal{E} = \{s=0\}$ if functions $\{h_i\}$ are chosen to be smooth functions, since these compactifications do not include any radicals of $s$.

\subsection{Periodic blow-up}
\label{section-periodic}

Theorem \ref{thm-stationary-blowup} shows that trajectories on stable manifolds of hyperbolic equilibria at infinity correspond to blow-up solutions, which are solutions we usually refer to as blow-ups.
From the viewpoint of dynamical systems, one expects that stable manifolds of {\em hyperbolic invariant sets at infinity} also characterize blow-up solutions.
The following theorem shows that this expectation is true for periodic orbits. 

\begin{thm}[Periodic blow-up]
\label{thm-periodic-blowup}
Assume that the desingularized vector field $g$ in (\ref{ODE-desing}) associated with (\ref{ODE-original}) is $C^1$ on $\overline{\mathcal{D}}$.
Suppose that $g$ admits a periodic orbit $\gamma_\ast = \{x_{\gamma_\ast}(\tau)\} \subset \mathcal{E}$, with period $T_\ast > 0$, characterized by a fixed point of the Poincar\'{e} map $P:\Delta \cap \overline{\mathcal{D}}\to \Delta \cap \overline{\mathcal{D}}$.
Let $x_\ast\in \Delta\cap \gamma_\ast$; namely, $P(x_\ast)=x_\ast$.
We further assume that all eigenvalues of Jacobian matrix $JP(x_\ast)$ have moduli away from $1$
(namely, $\gamma_\ast$ is hyperbolic), at least one of which has the modulus less than $1$.
\par
Then the solution $y(t)$ of (\ref{ODE-original}) whose image $x = T_{qP}(y)$ is on $W^s(\gamma_\ast)$ for $g$ satisfies $t_{\max} < \infty$; namely, $y(t)$ is a blow-up solution.
Moreover,
\begin{equation*}
p(y(t)) \sim c(t_{\max} - t)^{-1/k}\quad \text{ as }\quad t\to t_{\max},
\end{equation*}
where $k+1$ is the order of asymptotically quasi-homogeneous vector field $f$.
Finally, if the $i$-th component $(x_\ast)_i$ of $x_\ast$ with $i\in I_\alpha$ is not zero, then we also have
\begin{equation*}
y_i(t) \sim c(t_{\max} - t)^{-\alpha_i /k}x_i(-c'\ln (t_{\max} - t))\quad \text{ as }\quad t\to t_{\max}
\end{equation*}
for some constants $c\in \mathbb{R}$ and $c' > 0$.
\end{thm}

\begin{proof}
First note that, by Hartman-Grobman's theorem for hyperbolic periodic orbits (e.g., \cite{I1970, PS1970, Rob}), the (general $C^1$) flow $\varphi^\tau$ near a hyperbolic periodic orbit $\gamma$ is topologically conjugate\footnote{
For general periodic orbits $\gamma$, the Hartman-Grobman's theorem gives the topological equivalence between flows, namely, transformations between two flows may permit {\em re-parameterizations of time variable}.
The re-parameterizations may change the blow-up time and rate of solutions.
} to the following linear bundle (skew-product) flow on the normal bundle $\mathcal{N}_\gamma$ in a neighborhood of $\gamma \times \{0\}$:
\begin{equation*}
\Psi^\tau(q,v) = (\varphi^\tau (q), J\varphi^\tau(q)v),\quad q\in \gamma_\ast,\ v\in N_q^{u}(\gamma_\ast)\oplus N_q^{s}(\gamma_\ast),
\end{equation*}
where $\mathbb{R}^n = N_q^{u}(\gamma_\ast)\oplus T_q(\gamma_\ast)\oplus N_q^{s}(\gamma_\ast)$ is the $J\varphi^\tau(q(\tau))$-invariant splitting continuously depending on $q\in \gamma_\ast$.
Consider the flow $\varphi^\tau$ of $g$ in (\ref{ODE-desing}) around $\gamma_\ast$ and associated $\Psi^\tau$.
The stable manifold of $\gamma_\ast \times \{0\}$ for $\Psi^\tau$ is characterized by 
\begin{equation*}
\{(q, v) \mid q\in \gamma_\ast, v = 0\oplus v_s \in N_q^{u}(\gamma_\ast)\oplus N_q^{s}(\gamma_\ast)\}.
\end{equation*}
Let $\pi_{\gamma_\ast} : \mathcal{N}_{\gamma_\ast} \to \gamma_\ast$ be the natural projection $(q,v)\mapsto q$.
Then, for any solution $z(\tau)\in T_{\gamma_\ast} W^s(\gamma_\ast)$ of $\Psi^\tau$, there is a point $q_z\in \gamma_\ast$ such that $\pi_{\gamma_\ast}(z(0)) = q_z$.
Let $\pi_{\mathcal{N}} = id-\pi_{\gamma_\ast}$. 
Note that the flow $\Psi^\tau$ is the solution of the system
\begin{equation*}
\begin{cases}
\displaystyle{\frac{dq}{d\tau} = f(q)}, & \\
\displaystyle{\frac{dv}{d\tau} = (\pi_\mathcal{N}\circ J\varphi^\tau(q(\tau)))v}, & 
\end{cases}
\quad q(\tau)\in \gamma_\ast,\quad v(0) = v_0 = 0\oplus v_{0,s} \in N_{q(0)}^{u}(\gamma_\ast)\oplus N_{q(0)}^{s}(\gamma_\ast).
\end{equation*}
Note that the splitting $\bigcup_{q\in \gamma_\ast} (N_q^{u}(\gamma_\ast)\oplus N_q^{s}(\gamma_\ast))$ is $J\varphi^\tau(q(\tau))$-invariant for all $\tau$
and that the coefficient matrix $J\varphi^\tau(q(\tau))$ is $T_\ast$-periodic.
Thus the Floquet theory indicates that there is a $T_\ast$-periodic nonsingular matrix $S(\tau)$ and a matrix $R$ such that
\begin{equation*}
v(\tau) = S(\tau) e^{\tau R}v_0,\quad {\rm Spec}(e^{T_\ast R}) = \{\lambda_1,\cdots, \lambda_{n-1}\}.
\end{equation*}
Under the change of coordinates, we may assume to express
\begin{equation}
\label{sol-normal-diag}
u(\tau) = Q(\tau) e^{\tau \Lambda}u_0,\quad e^{T_\ast \Lambda} = {\rm diag}(J(\lambda_1;m_1),\cdots, J(\lambda_k;m_k)),\footnote{
This expression is realized by using a nonsingular matrix $C_R$ obtaining the Jordan normal form of $R$ and setting $u(\tau) = C_Rv(\tau)$ as well as $u_0 = C_Rv_0$, $Q(\tau) = C_RS(\tau)$.
By our assumption of $v_0$, only eigenvalues $\lambda_i$ satisfying $|\lambda_i| \leq \mu_\ast < 1$ involves the time evolution of (\ref{sol-normal-diag}).
}
\end{equation}
with $\sum_{j=1}^k m_j = n-1$.
Thus, letting $\mu_\ast$ a positive number satisfying $\mu_\ast = \max_{i=1,\cdots,n \text{ with }|\lambda_i| < 1}|\lambda_i|$, we have
\begin{equation*}
|u(\tau)| = C_1\tau^{m_\ast - 1} e^{\tau l} (1 + o(1)),
\end{equation*}
where $l\leq \ln \mu_\ast < 0$ and $m_\ast$ denotes the maximal dimension of Jordan block matrix of $\lambda_i$'s attaining $\mu_\ast = |\lambda_i|$.
Let $z_{\gamma_\ast}(\tau)$ be the solution with $\pi_{\gamma_\ast}z_{\gamma_\ast}(\tau) = q_z(\tau)\in \gamma_\ast$ and $\pi_{\mathcal{N}}z_{\gamma_\ast}(\tau) \equiv 0$ for $\tau \geq 0$.
Then we have
\begin{equation*}
|z(\tau) - z_{\gamma_\ast}(\tau)| = C_2\tau^{m_\ast - 1} e^{\tau l} (1 + o(1)).
\end{equation*}

Consequently, via the conjugacy $h = id + v$ with bounded continuous function $v$, we have
\begin{equation*}
|x(\tau) - x_{\gamma_\ast}(\tau)| = | h(z(\tau)- z_{\gamma_\ast}(\tau))| = C_3 \tau^{m_\ast - 1} e^{-m\tau}(1+o(1))\quad \text{ as }\tau \to \infty
\end{equation*}
with $-m \leq l$ and some constant $C_3$.
\par
Turn to the quantity $1-p^{2c}(x(\tau))$.
Note that $1-p^{2c}(x) = |1-p^{2c}(x)|$ for any $x\in \overline{\mathcal{D}}$.
Near the periodic orbit $\gamma_\ast = \{x_\gamma(\tau)\}$, we have
\begin{align}
\notag
|1-p^{2c}(x(\tau))| &= \left| 1- \sum_{i\in I_\alpha} a_i (x_i(\tau)- x_{\gamma,i}(\tau) + x_{\gamma,i}(\tau))^{2\beta_i} \right|\\
\notag
	&= \left| 1-\sum_{i\in I_\alpha} \sum_{k_i = 1}^{2\beta_i} a_i {}_{2\beta_i}C_{k_i} x_{\gamma,i}(\tau)^{2\beta_i - k_i}(x_i(\tau)- x_{\gamma,i}(\tau))^{k_i} \right| \\
\label{asym-radius-per}
	&= C_4 \tau^{m_\ast - 1} e^{-\ell \tau}(1+o(1))
\end{align}
for some $-\ell \leq -m < 0$ and a constant $C_4$, where the terms $1$ and $\sum_{i\in I_\alpha} a_i x_{\gamma,i}(\tau)^{2\beta_i}$ are cancelled out for all $\tau$.
Recall that we have for a certain initial point $t_0$,
\begin{align*}
t-t_0 &= \int_0^\tau \frac{d\eta}{\kappa(y(\eta))^{k}}\\
	&= \int_0^\infty \frac{d\eta}{\kappa(y(\eta))^{k}} - \int_\tau^\infty \frac{d\eta}{\kappa(y(\eta))^{k}}
\end{align*}
and the integrals converge due to (\ref{asym-radius-per}).
In particular, $t_{\max} < \infty$ holds and the solution is a blow-up solution.
Thus
\begin{equation*}
t_{\max} - t = \int_\tau^\infty \frac{d\eta}{\kappa(y(\eta))^{k}} = C_5 \tau^{k(m_\ast - 1)/2c} e^{- (k\ell / 2c) \tau}(1+o(1))\quad \text{ as } \tau \to \infty.
\end{equation*}
Since $dt/d\tau > 0$, this relation is then solvable for $\tau$ and, together with (\ref{asym-radius}), it yields that
\begin{equation*}
1-p^{2c}(x(\tau)) \sim C_6(t_{\max}-t)^{2c/k}
\end{equation*}
and
\begin{equation}
\label{rate-infty}
p(y(t)) = p(x(\tau))\kappa(x(\tau)) = \frac{p(x(\tau))}{(1-p(x(\tau))^{2c})^{1/2c}} \sim c(t_{\max} -t)^{-1/k}
\end{equation}
as $t\to t_{\max}$, where $C_5, C_6$ and $c$ are certain constants.
\par
It immediately follows from the above asymptotics that, for $y_i$ tending to $x_\ast$ with $(x_\ast)_i \not = 0$,
\begin{equation*}
y_i(t) = \kappa(y(t))^{\alpha_i} x_i(\tau) \sim p(y(t))^{\alpha_i} x_i(\tau)\quad \text{ as }t\to t_{\max} .
\end{equation*}
Now the time-scale desingularization $d\tau = \kappa(y(t))^k dt$ indicates
\begin{equation*}
\tau = \tau_0 + \int_{t_0}^t \kappa(y(\tilde t))^k d\tilde t \sim \tau_0 + c\int_{t_0}^t (t_{\max} -\tilde t)^{-1} d\tilde t\quad \text{ as }t,t_0 \to t_{\max}
\end{equation*}
by (A1) in Proposition \ref{prop-adm} and (\ref{rate-infty}).
The last integral is $-c\ln ((t_{\max}-t)/(t_{\max}-t_0)) = -c'\ln (t_{\max}-t)$.
Thus we obtain
\begin{equation*}
y_i(t) \sim p(y(t))^{\alpha_i} x_i(\tau)\sim c(t_{\max} -t)^{-\alpha_i/k} x_i(-c'\ln (t_{\max}-t))\quad \text{ as }t\to t_{\max}
\end{equation*}
and complete the proof.
\end{proof}

\begin{rem}
By the In-Phase Property of hyperbolic periodic orbits (e.g., \cite{Rob}), it further follows that the blow-up solution $y(t) = T_{qP}^{-1}(x(t))$ possesses the following property: there is a point $z_0\in \gamma$ wth the solution $z(\tau)$ of (\ref{ODE-desing}) such that $d(x(\tau), z(\tau)) \to 0$ as $\tau \to \infty$, in which sense periodic blow-up solution $y(t)$ behaves in-phase.
\end{rem}

In addition to the convergence of norms, the {\em in-phase property} of invariant sets at infinity is required for precise description of blow-up behavior when invariant sets at infinity themselves have nontrivial behavior.
We leave general cases such as non-hyperbolic periodic trajectories or general invariant sets at infinity to the future works, since they are beyond our current aims.

\section{Demonstration 1}
\label{section-Lienard}
In the rest of successive sections, we demonstrate several blow-up solutions as applications of our arguments.
First we go back to the polynomial Li\'{e}nard equation (\ref{Lienard}).
According to e.g. \cite{DH1999}, periodic orbits at infinity can be seen in polynomial Li\'{e}nard equation 
\begin{equation}
\label{Lienard-again}
\begin{cases}
\dot x = y, & \\
\dot y = -(\epsilon x^m + \sum_{k=0}^{m-1}a_k x^k) - y(x^n + \sum_{k=0}^{n-1}b_k x^k) & 
\end{cases}
\end{equation}
with appropriate degree\footnote{
Some of periodic orbits at infinity are shown to be hyperbolic.
However, there are no arguments from the viewpoint of blow-up solutions.
}, 
where $\epsilon = \pm 1$ if $m\not = 2n+1$, and $\epsilon \in \mathbb{R}\setminus \{0\}$ if $m=2n+1$.
As an example, consider (\ref{Lienard-again}) with the type $(2n+1, n)$.
The aim of this section is to briefly review the preceding results in a simple case as a nontrivial example generating periodic blow-up solutions. 
\par
First we immediately know the following property.
\begin{lem}
The system (\ref{Lienard-again}) with $m=2n+1$ is asymptotically quasi-homogeneous with type $(1,n+1)$ and order $n+1$.
\end{lem}
The vector field (\ref{Lienard-again}) thus associates the vector field (\ref{ODE-desing}) via the quasi-Poincar\'{e} compactification of type $(1,n+1)$ and the time-variable desingularization $d\tau = \kappa^n dt$.
Note that the correspondence matches the transformation discussed in Section 2.2 in \cite{DH1999}.
\par
Arguments in \cite{DH1999} show that, if $m=2n+1$ and $n$ is even, the system (\ref{Lienard-again}) possesses limit cycles at infinity which are repelling.
By Theorem \ref{thm-periodic-blowup}, the limit cycle would induce blow-up solutions in backward-time flow.
\par
To confirm this observation, consider the following system
\begin{equation}
\label{Lienard-simple}
\begin{cases}
y_1' = y_2, & \\
y_2' = -y_1^{2n+1} - y_1^n y_2. & 
\end{cases}
\end{equation}
Here we use the $(1,n+1)$-polar coordinate $(r,\theta)$ given by
\begin{equation*}
y_1 = \frac{{\rm Cs}\theta}{r},\quad y_2 = \frac{{\rm Sn}\theta}{r^{n+1}}\quad\text{ with }\quad {\rm Cs}^{2n+2}\theta + (n+1){\rm Sn}^2 \theta = 1.
\end{equation*}
We only note that
\begin{equation}
\label{cs-sn-even}
\int_0^T {\rm Cs}^k \theta {\rm Sn}^2 \theta d\theta > 0\quad \text{ if $k$ is even}.
\end{equation}
In the coordinate $(r, \theta)$, we have
\begin{align*}
r' &= - r^{2n+3}(y_1^{2n+1}y_1' + y_2 y_2') = - r^{2n+3}\{y_1^{2n+1}y_2 + y_2 (-y_1^{2n+1} - y_1^n y_2)\}\\
	&= r^{2n+3} y_1^n y_2^2 = r^{-(n-1)}{\rm Cs}^n \theta {\rm Sn}^2 \theta,\\
\theta' &= -(n+1)r {\rm Sn}\theta y_1' + r^{n+1} {\rm Cs}\theta y_2' = -(n+1)r {\rm Sn}\theta y_2 + r^{n+1} {\rm Cs}\theta (-y_1^{2n+1} - y_1^n y_2)\\
	&= -(n+1)r^{-n} {\rm Sn}^2\theta + r^{n+1} {\rm Cs}\theta (- {\rm Cs}^{2n+1}\theta r^{-(2n+1)} - {\rm Cs}^{n}\theta {\rm Sn}\theta r^{-(2n+1)})\\
	&= -r^{-n} (1 + {\rm Cs}^{n+1}\theta  {\rm Sn}\theta).
\end{align*}
Using the time-variable desinglarization $d\tau = r^{-n} dt$, we have
\begin{equation}
\label{Lienard-simple-desing}
\frac{dr}{d\tau} = r{\rm Cs}^n \theta {\rm Sn}^2 \theta,\quad \frac{d\theta}{d\tau} = -(1 + {\rm Cs}^{n+1}\theta  {\rm Sn}\theta).
\end{equation}
Observe that $r = 0$ satisfies $dr/d\tau = 0$ for any $\theta$. 
Moreover, $|{\rm Cs}^{n+1}\theta {\rm Sn}\theta| \leq ({\rm Cs}^{2(n+1)}\theta + {\rm Sn}^2\theta)^{1/2} \leq ({\rm Cs}^{2(n+1)}\theta + (n+1){\rm Sn}^2\theta)^{1/2} = 1$ and, since ${\rm Cs}0 = 1$, then $d\theta / d\tau$ never vanishes.
As a consequence, we have that the horizon $\{r = 0\}$ is an invariant periodic orbit.
Dynamics near $\{r=0\}$ is thus reduced to the following regular system:
\begin{equation*}
\frac{dr}{d\theta} = -\frac{{\rm Cs}^n \theta {\rm Sn}^2 \theta}{1 + {\rm Cs}^{n+1}\theta  {\rm Sn}\theta}r.
\end{equation*}
Then the Poincar\'{e} map $P$ on a section $\{0\leq r \leq \epsilon, \theta = 0\}$ with small $\epsilon > 0$ is
\begin{equation*}
P(r_0) = e^{\alpha(T)}r_0,\quad \alpha(T) = -\int_0^T \frac{{\rm Cs}^n \theta {\rm Sn}^2 \theta}{1 + {\rm Cs}^{n+1}\theta  {\rm Sn}\theta} d\theta.
\end{equation*}
By the fact $0 < C_1 \leq 1 + {\rm Cs}^{n+1}\theta  {\rm Sn}\theta \leq C_2 < \infty$ for $\theta\in [0,T]$ and (\ref{cs-sn-even}), we know $\alpha(T) < 0$.
Therefore the limit cycle at infinity $\{r=0\}$ is hyperbolic and repelling\footnote
{
Although $e^{\alpha(T)} < 1$, the variable $\theta$ varies in negative direction and hence the stability is totally reverse against the apparent calculation result.
}.
Theorem \ref{thm-periodic-blowup} and the above calculations yield the following result.


\begin{thm}
Consider (\ref{Lienard-simple}) with the backward-time direction.
Assume that $n$ is even.
Then any divergent solution $y(t) = (y_1(t), y_2(t))$ is periodic blow-up solution with the blow-up rate
\begin{equation*}
\begin{cases}
y_1(t) \sim c_1(t_{\max} - t)^{-1/n}x_1(-c'\ln (t_{\max}-t)), & \\
y_2(t) \sim c_2(t_{\max} - t)^{-(n+1)/n}x_2(-c'\ln (t_{\max}-t)) & \\
\end{cases}
\quad \text{ as }\quad t\to t_{\max}.
\end{equation*}
The periodic blow-up behavior with $n=2$ is described in Figure \ref{fig-Lienard}.
\end{thm}

This theorem can be generalized to (\ref{Lienard-again}) with $m=2n+1$ and even $n$ by the same arguments in \cite{DH1999}, but we omit the detail.


\begin{figure}[htbp]\em
\begin{minipage}{0.33\hsize}
\centering
\includegraphics[width=5cm]{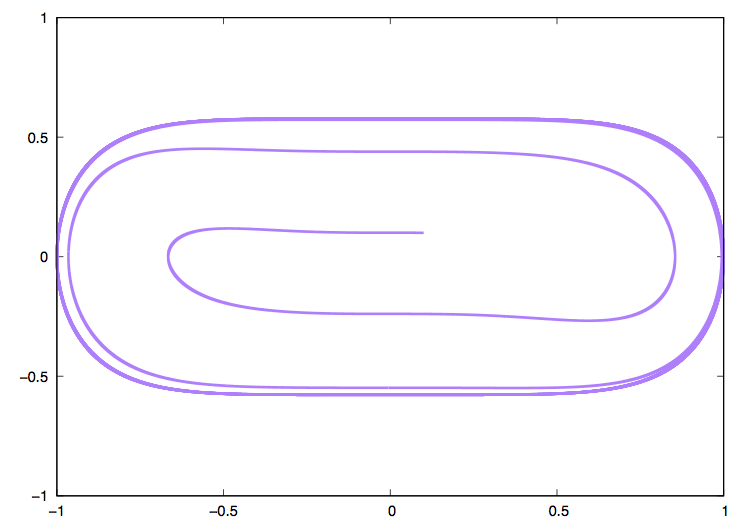}
(a)
\end{minipage}
\begin{minipage}{0.33\hsize}
\centering
\includegraphics[width=5cm]{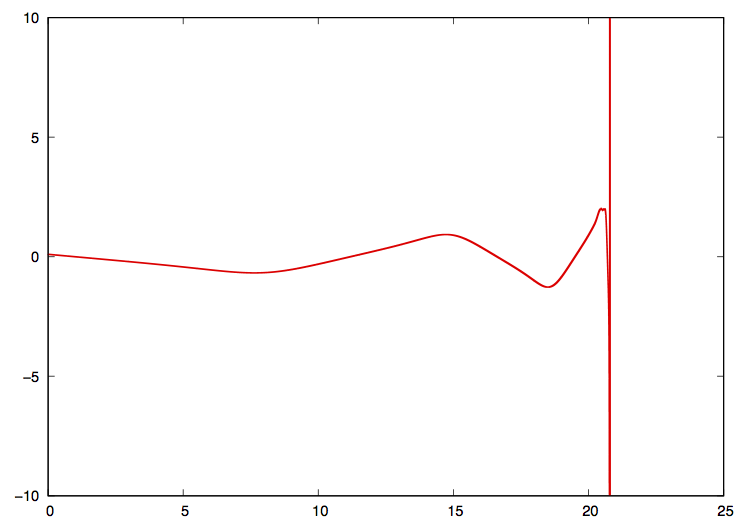}
(b)
\end{minipage}
\begin{minipage}{0.33\hsize}
\centering
\includegraphics[width=5cm]{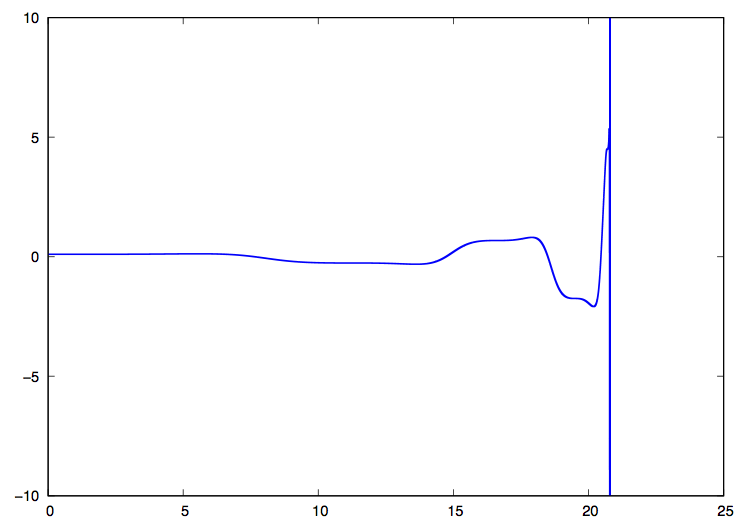}
(c)
\end{minipage}
\caption{Periodic blow-up : a solution of (\ref{Lienard-simple-desing}) with $n=2$ in backward-time direction}
\label{fig-Lienard}
(a) : a solution of (\ref{Lienard-simple-desing}) with $n=2$ in backward-time direction.
Coordinate is the orthogonal $(x_1,x_2)$-coordinate with $(a_1,a_2) = (1,3)$ instead of polar $(r,\theta)$-coordinate for simple numerical calculations.
The corresponding vector field (in backward-time direction) is
\begin{equation*}
\dot x_1 = -x_2 - x_1^{n+1} x_2^2,\quad
\dot x_2 = x_1^{2n+1} + x_1^n x_2 - (n+1) x_1^n x_2^3. 
\end{equation*}
The initial data is set as $(x_1, x_2) = (0.1,0.1)$. 
(b) : the $(t,y_1)$-plot of solution (a).
(c) : the $(t,y_2)$-plot of solution (a).
The solution blows up at $t_{\max}\sim 20.785$.
\end{figure}

\section{Demonstration 2}
\label{section-sing-shock}
Next we consider the following system of ODEs:
\begin{equation}
\label{KK-ex}
\begin{cases}
u' = u^2 - v, & \\
v' = \frac{1}{3}u^3. &
\end{cases}
\end{equation}
The system (\ref{KK-ex}) is well-known as the quasi-homogeneous part of traveling wave (viscous shock) equation derived from the {\em Keyfitz-Kranser model}  \cite{KK1990, KK1995}, which is the following initial value problem of the system of conversation laws:
\begin{equation}
\label{KK-PDE}
\begin{cases}
\displaystyle
{
\frac{\partial u}{\partial t} + \frac{\partial }{\partial x}(u^2 - v) = 0,
} & \\
\displaystyle
{
\frac{\partial v}{\partial t} + \frac{\partial }{\partial x}\left(\frac{1}{3}u^3 - u\right)=0,
} &
\end{cases}
\quad 
(u(x,0), v(x,0)) = 
\begin{cases}
(u_L, v_L) & x<0, \\
(u_R, v_R) & x>0.
\end{cases}
\end{equation}

The aim of this section is to discuss the application of quasi-Poincar\'{e} compactifications to (\ref{KK-ex}) for observing blow-up solutions and related singularity aspect of solutions with the help of numerical simulations.
We see that quasi-Poincar\'{e} compactifications give us comprehensive observations of blow-up solutions including unstable stationary blow-up solutions.

\subsection{Compactification and desingularization}
Let $f(u,v) = (f_1(u,v), f_2(u,v))$ be $f_1(u,v) = u^2 - v$ and $f_2(u,v) = \frac{1}{3}u^3$. 
Then we immediately have the following observation.
\begin{lem}
\label{lem-KK-type}
The vector field $f$ is (asymptotically) quasi-homogeneous of type $(1,2)$ and order $2$.
\end{lem}

We apply the quasi-Poincar\'{e} compactification of type $(1,2)$ with $a_1 = 1, a_2 = 2$ in Definition \ref{dfn-qP} as follows:
\begin{equation}
\label{coord-KK-ortho}
x_1 = \frac{u}{\kappa},\quad x_2 = \frac{v}{\kappa^2},\quad \kappa = \kappa(u,v) = (1+u^4 + 2v^2)^{1/4},
\end{equation}
as well as its quasi-polar coordinate representations
\begin{equation}
\label{coord-KK-polar}
u = \frac{{\rm Cs}\theta}{r},\quad v = \frac{{\rm Sn}\theta}{r^2},\quad r = \kappa^{-1}.
\end{equation}
Note that this quasi-Poincar\'{e} compactification is exactly same as the scaling of solutions near infinity in \cite{SSS1993}.
\par
The desingularized vector field of order $k +1 = 2$ for (\ref{KK-ex}) with orthogonal coordinate (\ref{coord-KK-ortho}) is
\begin{equation}
\label{KK-ex-desing}
\begin{cases}
\displaystyle
{
\dot x_1 = (x_1^2 - x_2) - x_1\left\{ x_1^3 (x_1^2 - x_2) + \frac{1}{3}x_1^3x_2 \right\}, 
}& \\
\displaystyle
{
\dot x_2 = \frac{1}{3}x_1^3 - 2x_2\left\{ x_1^3 (x_1^2 - x_2) + \frac{1}{3}x_1^3x_2 \right\},
}
& \\
\end{cases} \dot{} = \frac{d}{d\tau}.
\end{equation}
Similarly, the desingularized vector field of order $k +1 = 2$ for (\ref{KK-ex}) with quasi-polar coordinate (\ref{coord-KK-polar}) is
\begin{equation}
\label{KK-ex-desing-polar}
\begin{cases}
\displaystyle
{
\dot r = -r\left( {\rm Cs}^5 \theta - \frac{2}{3}{\rm Cs}^3\theta {\rm Sn}\theta \right), 
}& \\
\displaystyle
{
\dot \theta = -2{\rm Cs}^2\theta {\rm Sn}\theta + 2{\rm Sn}^2 \theta + \frac{1}{3} {\rm Cs}^4 \theta,
}
& \\
\end{cases} \dot{} = \frac{d}{d\tau_d}.
\end{equation}

We consider all vector fields depending on situations.
More precisely, we consider (\ref{KK-ex-desing}) for calculations of equilibria at infinity and numerical simulations of global trajectories, while we consider (\ref{KK-ex-desing-polar}) for calculating eigenvalues of Jacobian matrix at these points.

\subsection{Dynamics at infinity}
Here we consider equilibria of (\ref{KK-ex-desing}) on the horizon $\mathcal{E} = \left\{p(x)= (1+x_1^4 + 2x_2^2)^{1/4}=1\right\}$.
They should satisfy
\begin{equation}
\label{KK-ex-equilibrium}
\begin{cases}
\displaystyle
{
(x_1^2 - x_2) - \frac{x_1}{2}\left\{ 2x_1^3 (x_1^2 - x_2) + \frac{2}{3}x_1^3x_2 \right\} = 0, 
}& \\
\displaystyle
{
\frac{1}{3}x_1^3 - x_2\left\{ 2x_1^3 (x_1^2 - x_2) + \frac{2}{3}x_1^3x_2 \right\} = 0, 
}
& \\
x_1^4 + 2x_2^2 = 1\quad (\Leftrightarrow\ p(x) = 1). &
\end{cases}
\end{equation}
We immediately know that, if $(x_1, x_2)$ is an equilibrium at infinity, then $x_1, x_2\not = 0, \pm 1$.
\par
We may thus divide the second equation by $x_1^3$ to obtain $1 = 6x_2x_1^2 - 4x_2^2$.
We thus have
\begin{equation}
\label{KK-cp}
x_1^2  = \frac{1 + 4x_2^2}{6x_2}.
\end{equation}
Since $x_1^2 > 0$, $x_2$ has to be positive.
Moreover, since $x_1^2 < 1$, then $x_2$ also has to satisfy $1 + 4x_2^2 < 6x_2$.
Thus we have $(3-\sqrt{5})/4 < x_2 < 1/\sqrt{2}$.
\par
The first equation of (\ref{KK-ex-equilibrium}) is 
\begin{align*}
(x_1^2 - x_2) - &\frac{x_1}{2}\left\{ 2x_1^3 (x_1^2 - x_2) + \frac{2}{3}x_1^3x_2 \right\} = 0 \Leftrightarrow (x_1^2 - x_2) - x_1^4 (x_1^2 - x_2) - \frac{1}{3}x_1^4x_2 =0\\
	&\Leftrightarrow 2x_2^2(x_1^2 - x_2) - \frac{1}{3}(1-2x_2^2) x_2 =0\quad (\text{by }x_1^4 + 2x_2^2 = 1).
\end{align*}
We may divide the right-most side by $x_2$ to obtain $6x_1^2 x_2 - 1 - 4x_2^2 = 0$, which is exactly same as (\ref{KK-cp}).
We substitute (\ref{KK-cp}) into $p(x) = 1$ to obtain
\begin{align*}
\left(\frac{1 + 4x_2^2}{6x_2}\right)^2 + 2x_2^2 = 1 &\Leftrightarrow (1+4x_2^2)^2 = 36x_2^2(1-2x_2^2)\\
	&\Leftrightarrow 1 + 8\lambda + 16\lambda^2 = 36\lambda - 72\lambda^2 \quad (\text{setting }\lambda \equiv x_2^2)\\
	&\Leftrightarrow \lambda = \frac{7 \pm 3\sqrt{3}}{44} \quad \left(\text{which satisfy }\lambda= x_2^2 > \left(\frac{3-\sqrt{5}}{4}\right)^2\right).
\end{align*}
Thus we have
\begin{equation*}
x_2 = \sqrt{\frac{7+ 3\sqrt{3}}{44}} \approx 0.52648388611,\quad  \sqrt{\frac{7 - 3\sqrt{3}}{44}} \approx 0.20247601301\in \left(\frac{3-\sqrt{5}}{4}, \frac{1}{\sqrt{2}}\right),
\end{equation*}
as well as
\begin{equation*}
x_1 = \pm \sqrt[4]{\frac{15 - 3\sqrt{3}}{22}} \approx \pm 0.81704027943,\quad \pm \sqrt[4]{\frac{15 + 3\sqrt{3}}{22}} \approx \pm 0.97883950723.
\end{equation*}

As a consequence, we obtain the following result.
\begin{lem}
\label{lem-cp-infty-KK}
Equilibria at infinity of (\ref{KK-ex}) associated with quasi-Poincar\'{e} compactification with type $(1,2)$ and order $2$ are the following four points:
\begin{equation*}
p_1^\pm = \left(\pm \sqrt[4]{\frac{15 + 3\sqrt{3}}{22}}, \sqrt{\frac{7 - 3\sqrt{3}}{44}}\right),\quad p_2^\pm = \left(\pm \sqrt[4]{\frac{15 - 3\sqrt{3}}{22}}, \sqrt{\frac{7 + 3\sqrt{3}}{44}}\right).
\end{equation*}
\end{lem}
Next we consider the dynamics of (\ref{KK-ex-desing}) on $\mathcal{E}$, which is the dynamics on a simple closed curve (namely, no self-crossing points) with four equilibria.
Note that $\mathcal{E}$ is an invariant submanifold of $\overline{\mathcal{D}}$ for (\ref{KK-ex-desing}) from Theorem \ref{thm-dyn-infty}-1.
Dynamics on $\mathcal{E}$ behaves monotonously off four equilibria stated in Lemma \ref{lem-cp-infty-KK}.
In particular, the vector field (\ref{KK-ex-desing}) at $(x_1, x_2) = (0,\pm 1)$ is
\begin{equation*}
\dot x_1\mid_{(x_1, x_2) = (0,\pm 1)} = \mp 1,\quad \dot x_2\mid_{(x_1, x_2) = (0,\pm 1)} = 0,
\end{equation*}
Similarly, we have
\begin{equation*}
\dot x_1\mid_{(x_1, x_2) = (\pm \sqrt[4]{1/2}, 1/2)} = \frac{1}{\sqrt{2}} - \frac{1}{3} > 0,\quad \dot x_2\mid_{(x_1, x_2) = (\pm \sqrt[4]{1/2}, 1/2)} = \frac{\pm 1}{\sqrt[4]{2}}\left(\frac{\sqrt{2}}{3} - \frac{1}{2}\right) < 0.
\end{equation*}

Consequently, we have the following result, which can be also obtained for quasi-polar coordinates.
\begin{prop}
There are heteroclinic orbits for (\ref{KK-ex-desing}) in $\mathcal{E}$ from (i) : $p_1^-$ to $p_1^+$, (ii) : $p_2^+$ to $p_1^+$, (iii) $p_2^+$ to $p_2^-$, and (iv) $p_1^-$ to $p_2^-$.
The closure of these orbits fulfills $\mathcal{E}$.
\end{prop}

\subsection{Blow-up solutions}
We calculate the Jacobian matrix of (\ref{KK-ex-desing}) as well as (\ref{KK-ex-desing-polar}) at $p_i^\pm$.
Both vector fields (\ref{KK-ex-desing}) and (\ref{KK-ex-desing-polar}) are $C^1$ on $\partial \mathcal{D}$.
According to Theorem \ref{thm-stationary-blowup}, equilibria at infinity $p_i^\pm$ induce blow-up solutions if they are hyperbolic.
We immediately have
\begin{equation}
\label{Jacobi-KK-ex}
Jg(r,\theta)|_{r=0} \equiv \frac{\partial g(r,\theta)}{\partial (r,\theta)}\left|_{r=0} \right.= \begin{pmatrix}
-{\rm Cs}^5\theta + \frac{2}{3}{\rm Cs}^3\theta {\rm Sn}\theta & 0\\
s {\rm Cs}\theta {\rm Sn}\theta & 4{\rm Cs}\theta {\rm Sn}^2\theta -2{\rm Cs}^5\theta + \frac{8}{3}{\rm Cs}^3\theta {\rm Sn}\theta
\end{pmatrix}.
\end{equation}
Then the eigenvalues $\mu_r$ and $\mu_\theta$ are
\begin{equation*}
\mu_r(\theta) = -{\rm Cs}^5\theta + \frac{2}{3}{\rm Cs}^3\theta {\rm Sn}\theta,\quad \mu_\theta(\theta) = 4{\rm Cs}\theta {\rm Sn}^2\theta -2{\rm Cs}^5\theta + \frac{8}{3}{\rm Cs}^3\theta {\rm Sn}\theta,
\end{equation*}
which describes the stability in the $r$-direction and the $\theta$-direction, respectively.
\begin{itemize}
\item At $p_1^+$,
\begin{align*}
\mu_r &= -\left(\frac{15+3\sqrt{3}}{22}\right)^{3/4} \left( \sqrt{\frac{15+3\sqrt{3}}{22}} - \frac{2}{3}\sqrt{\frac{7-3\sqrt{3}}{44}} \right)\approx -0.7719863801113,\\
\mu_\theta &= 2\left(\frac{15+3\sqrt{3}}{22}\right)^{1/4} \left(\frac{-4-3\sqrt{3}}{11} + \frac{4}{3} \sqrt{\frac{(15+3\sqrt{3})(7-3\sqrt{3})}{22\cdot 44}}\right)  \approx -1.130266505985.
\end{align*}
%
\item At $p_1^-$, $\mu_r(p_1^-) = -\mu_r(p_1^+)$ and $\mu_\theta(p_1^-) = -\mu_\theta(p_1^+)$ by symmetry.
\item At $p_2^+$, 
\begin{align*}
\mu_r &= -\left(\frac{15-3\sqrt{3}}{22}\right)^{3/4} \left( \sqrt{\frac{15-3\sqrt{3}}{22}} - \frac{2}{3}\sqrt{\frac{7+3\sqrt{3}}{44}} \right)\approx -0.1726609270826,\\
\mu_\theta &= 2\left(\frac{15-3\sqrt{3}}{22}\right)^{1/4} \left(\frac{-4+3\sqrt{3}}{11} + \frac{4}{3} \sqrt{\frac{(15-3\sqrt{3})(7+3\sqrt{3})}{22\cdot 44}}\right)  \approx +0.9434368505431.
\end{align*}
%
\item At $p_2^-$, $\mu_r(p_2^-) = -\mu_r(p_2^+)$ and $\mu_\theta(p_2^-) = -\mu_\theta(p_2^+)$ by symmetry.
\end{itemize}
That is, $p_1^+$ is a sink, $p_1^-$ is a source, and both $p_2^\pm$ are saddles.
Hyperbolicity is independent of the choice of compactifications\footnote{
Indeed, numerical computations for quasi-Poincar\'{e} compactifications indicate that these equilibria as those for (\ref{KK-ex-desing}) are hyperbolic and the same stability information as those for (\ref{KK-ex-desing-polar}).
}. 
Theorem \ref{thm-stationary-blowup} implies that trajectories of (\ref{KK-ex}) whose images are asymptotic to these equilibria in appropriate time directions, which are shown in Figure \ref{fig-KK}, are blow-up solutions.

\begin{figure}[htbp]\em
\begin{minipage}{1\hsize}
\centering
\includegraphics[width=7cm]{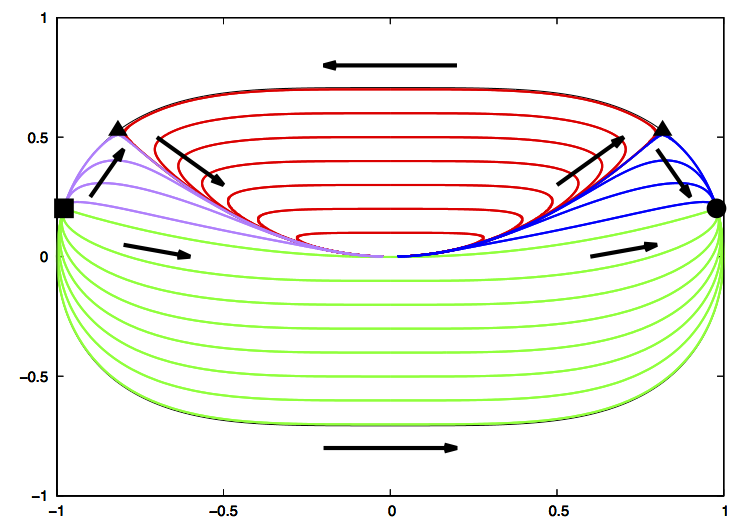}
\end{minipage}
\caption{Global trajectories of (\ref{KK-ex-desing}) on $\overline{\mathcal{D}}$}
\label{fig-KK}
(a) Red : global trajectories through $x_1= 0, x_2 > 0$.
(b) Green : global trajectories through $x_1= 0, x_2 < 0$.
(c) Blue : global trajectories connecting $p_1^-$ and the origin.
(d) Purple : global trajectories connecting the origin and $p_1^+$.
\par
All trajectories are computed by using standard Runge-Kutta explicit scheme, not advanced scheme for hyperbolic systems such as Godunov scheme for systems of conservation laws.
Trajectories (b), (c) and (d) are actually correspond to blow-up solutions, because all equilibria at infinity are hyperbolic.
\end{figure}

\section{Demonstration 3}
\label{section-two-fluid}
In this section, consider the Riemann problem of the following system of conservation laws derived from (simplified) two-phase, one-dimensional imcompressible flow \cite{KSS2003}:
\begin{equation}
\label{two-fluid-PDE}
\beta_t + (vB_1(\beta))_x = 0,\quad
v_t + (v^2 B_2(\beta))_x = 0
\end{equation}
with
\begin{equation}
\label{data-two-fluid-PDE}
(\beta(x,0), v(x,0)) = 
\begin{cases}
U_L \equiv (\beta_L, v_L) &\text{$x<0$},\\
U_R \equiv (\beta_R, v_R) &\text{$x>0$},
\end{cases}
\end{equation}
where
\begin{equation*}
B_1(\beta) = \frac{(\beta-\rho_1)(\beta-\rho_2)}{\beta},\quad B_2(\beta) = \frac{\beta^2- \rho_1\rho_2}{2\beta^2}
\end{equation*}
and $\rho_2 > \rho_1$ are positive constants.
Observe that $B_1(\beta) < 0$ for $\beta\in (\rho_1, \rho_2)$ and $B_1(\beta) > 0$ for $0< \beta < \rho_1, \beta > \rho_2$. 
Note that eigenvalues of the Jacobian matrix of $F(U) = (vB_1(\beta), v^2B_2(\beta))^T$ for $U=(\beta,v)$ is
\begin{equation*}
\lambda(U) = 2vB_2(\beta) \pm v\sqrt{B_1(\beta)B_2'(\beta)},\quad B_2'(\beta) = \frac{\rho_1 \rho_2}{\beta^3}
\end{equation*}
and have nonzero imaginary parts except when $\beta = \rho_i$ or $v=0$.
In particular, if $\beta$ is in the interior of {\em physical range} $\rho_1 \leq \beta \leq \rho_2$, then the system is not strictly hyperbolic, which leads to change several properties associated with characteristics. 
Details are stated in \cite{KSS2003}. 
\par
In contrast with (\ref{KK-PDE}), the variable $\beta$ is assumed to be bounded for solutions of (\ref{two-fluid-PDE}) from the physical viewpoint\footnote{
The constraint comes from the fact that $\beta$ is a linear combination of the volume fractions (the sum of these is always $1$) and the densities of phases ($\rho_1$ and $\rho_2$).
} and the applications of full compactifications (namely, nontrivial transformations for all variable) is not suitable for blow-up behavior.
Therefore, when we consider blow-up solutions of (\ref{two-fluid-PDE}) with the physical relevance, it is natural to consider the boundary value problem of ordinary differential equations {\em of the type containing $0$}:
\begin{equation}
\label{two-fluid}
\begin{cases}
\beta' = vB_1(\beta) - c\beta - c_1, & \\
v' =  v^2 B_2(\beta) - cv - c_2, & 
\end{cases}\quad {}'=\frac{d}{d\zeta},\quad \zeta = x-ct,
\end{equation}
\begin{equation*}
\lim_{\zeta \to -\infty}(\beta(\zeta), v(\zeta)) = (\beta_L, v_L),\quad \lim_{\zeta \to +\infty}(\beta(\zeta), v(\zeta)) = (\beta_R, v_R),
\end{equation*}
where $c$ is the speed of traveling waves 
\begin{equation*}
c = \frac{v_R B_1(\beta_R) - v_L B_1(\beta_L)}{\beta_R - \beta_L}
\end{equation*}
and $(c_1,c_2) = (c_{1L}, c_{2L})$ or $(c_{1R}, c_{2R})$ with
\begin{equation*}
\begin{cases}
c_{1L} = v_L B_1(\beta_L) - c\beta_L, & \\
c_{2L} = v_L^2 B_2(\beta_L) -cv_L, & \\
\end{cases}
\quad
\begin{cases}
c_{1R} = v_R B_1(\beta_R) - c\beta_R, & \\
c_{2R} = v_R^2 B_2(\beta_R) -cv_R. & \\
\end{cases}
\end{equation*}
The system (\ref{two-fluid}) is actually the traveling wave (viscous shock) equation of (\ref{two-fluid-PDE}) for $(\beta(t,x), v(t,x)) = (\hat \beta(\zeta), \hat v(\zeta))$.
Easy calculations yield the following property.
\begin{lem} 
\label{lem-AQH-two-fluid}
The vector field in the right-hand side of (\ref{two-fluid}) is defined in a neighborhood of $\{\rho_1\leq \beta \leq \rho_2\}\times \mathbb{R}$ and asymptotically quasi-homogeneous of type $(0,1)$ and order $2$.
\end{lem}

Following Lemma \ref{lem-AQH-two-fluid}, we choose the directional compactification $T$ of type $(0,1)$ : $(\beta, v)\mapsto (x_1, r) = (\beta, v^{-1})$.
Direct calculations yield the following desingularized vector field on $\{r\geq 0\}\times \{\rho_1\leq \beta \leq \rho_2\}$:
\begin{equation}
\label{two-fluid-desing}
\begin{cases}
\displaystyle{\frac{dx_1}{d\tau} = B_1(x_1) - cx_1 r - c_1r}, & \\
\displaystyle{\frac{dr}{d\tau} =  -r\left\{B_2(x_1) - cr - c_2r^2\right\}}, & 
\end{cases}
\end{equation}
where $\tau$ is the desingularized time-scale given by $d\tau = r^{-1}dt$.
Obviously, $(x_1, r) = (\rho_1,0)\equiv p_1$ and $(\rho_2, 0)\equiv p_2$ are equilibria of (\ref{two-fluid-desing}) on the horizon $\mathcal{E} = \{r=0\}$ and the vector field on $\mathcal{E}\setminus \{p_1,p_2\}$ is monotone on each component.
In the coordinate $(x_1,r)$, the desingularized vector field (\ref{two-fluid-desing}) is $C^1$ locally including $\mathcal{E}$, and hence the Jacobian matrices of (\ref{two-fluid-desing}) 
\begin{equation*}
Jg(x_1,r) = 
\begin{pmatrix}
\frac{dB_1}{dx_1} - cr & -cx_1 - c_1\\
-r\frac{dB_2}{dx_1} & -(B_2(x_1) - cr - c_2 r^2) -r\{-c -2 c_2r\}
\end{pmatrix}
\end{equation*}
at $p_1$ and $p_2$ make sense, which are 
\begin{equation*}
Jg(p_1) = 
\begin{pmatrix}
\frac{2\rho_1^2 - \rho_1(\rho_1+\rho_2)}{\rho_1^2} & -c\rho_1\\
0 & -B_2(\rho_1)
\end{pmatrix}, \quad 
Jg(p_2) = 
\begin{pmatrix}
\frac{2\rho_2^2 - \rho_2(\rho_1+\rho_2)}{\rho_2^2} & -c\rho_2\\
0 & -B_2(\rho_2)
\end{pmatrix}
\end{equation*}
and eigenvalues $\{\mu_1,\mu_2\}$ are
\begin{align*}
\mu_1(p_1) &= 2-\frac{(\rho_1 + \rho_2)}{\rho_1} < 0,\quad \mu_2(p_1) = -\frac{1}{2}\left(1-\frac{\rho_2}{\rho_1}\right) > 0,\\
\mu_1(p_2) &= 2-\frac{(\rho_1 + \rho_2)}{\rho_2} > 0,\quad \mu_2(p_2) = -\frac{1}{2}\left(1-\frac{\rho_1}{\rho_2}\right) < 0.
\end{align*}
These results indicate that both $p_1$ and $p_2$ are hyperbolic saddles.
Theorem \ref{thm-stationary-blowup} thus shows that trajectories on $W^s(p_2)$ (in forward time) and $W^u(p_1)$ (in backward time) correspond to blow-up solutions of (\ref{two-fluid}).
Phase portraits of (\ref{two-fluid-desing}) is shown in Figure \ref{fig-phase-two-fluid}.
The equation (\ref{two-fluid-desing}) contains the following sequence of heteroclinic orbits:
\begin{equation*}
W_1 : T(U_L)\to p_2,\quad W_2 : p_2\to p_1,\quad W_3 : p_1\to T(U_R).
\end{equation*}

\begin{figure}[htbp]\em
\begin{minipage}{1\hsize}
\centering
\includegraphics[width=7cm]{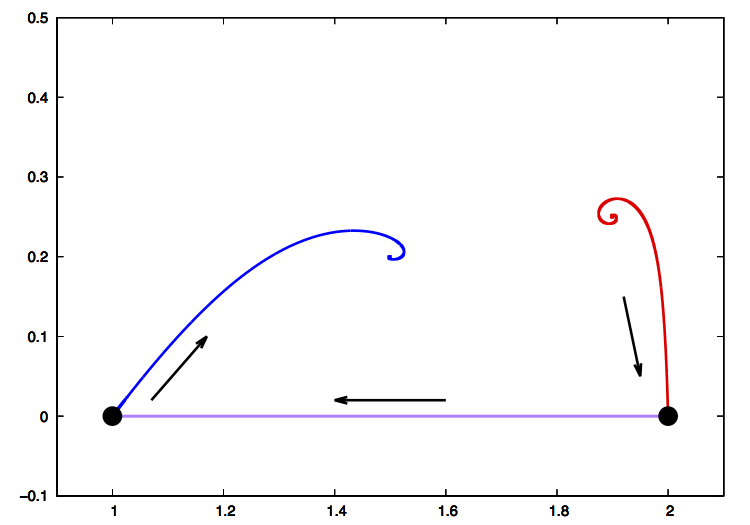}
\end{minipage}
\caption{The sequence $\{W_i\}_{i=1}^3$ for (\ref{two-fluid-desing}) in the coordinate $x = (x_1,r)$}
\label{fig-phase-two-fluid}
$W_1$ : a heteroclinic orbit connecting $x_L = T(U_L) = (1.9, 0.25)$ and $p_2$.
Note that $x_L$ is a source equilibrium of (\ref{KK-ex}) with $(c_1,c_2) = (c_{1L}, c_{2L})$.
$W_2$ : a heteroclinic orbit connecting $p_2$ and $p_1$.
$W_3$ : a heteroclinic orbit connecting $p_1$ and $x_R = T(U_R) = (1.5, 0.2)$.
Note that $x_R$ is a sink equilibrium of (\ref{two-fluid-desing}) with $(c_1,c_2) = (c_{1R}, c_{2R})$.
\end{figure}

\begin{figure}[htbp]\em
\begin{minipage}{0.5\hsize}
\centering
\includegraphics[width=7cm]{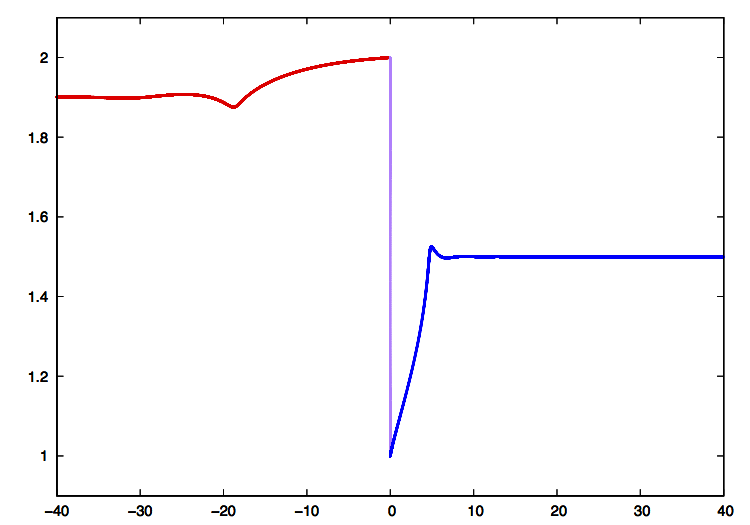}
(a)
\end{minipage}
\begin{minipage}{0.5\hsize}
\centering
\includegraphics[width=7cm]{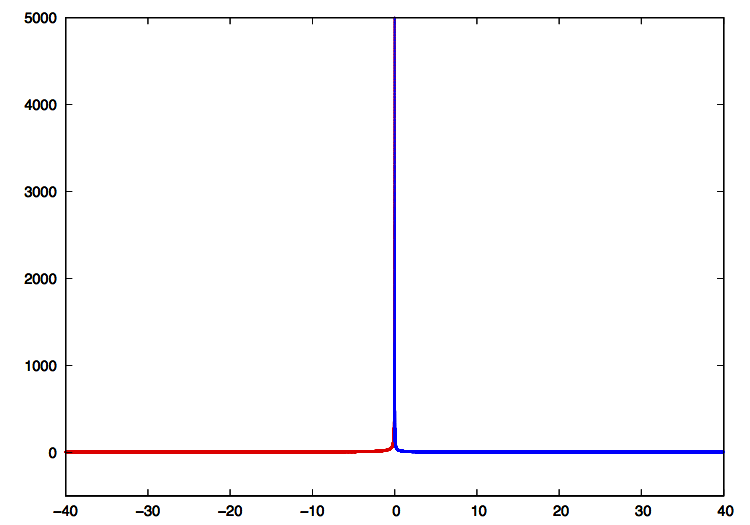}
(b)
\end{minipage}
\caption{Profiles of $(\beta(t,x),v(t,x)) = (\hat \beta\left(\frac{x-ct}{\epsilon r}\right), \hat v\left(\frac{x-ct}{\epsilon r}\right))$ for (\ref{two-fluid-PDE}) corresponding to $\{W_i\}_{i=1}^3$}
\label{fig-profile-two-fluid}
(a) : Profile of $\hat \beta(\frac{x-ct}{\epsilon r})$. (b) : Profile of $\hat v(\frac{x-ct}{\epsilon r})$.
The factor $r$ stems from the time-scale desingularization $d\tau = r^{-1} dt$.
Asymptotic behavior in the outer layer (namely, as $x-ct \to \mp \infty$) corresponds to $W_1$ and $W_3$, respectively.
Signatures of $\beta$ and $v$ in the inner layer (namely, near Dirac delta singularity) reflect transition of $W_2$.
These profiles transit in the positive $x$ direction with the speed $c \approx 1.60964912281$.
\end{figure}

\section{Conclusion}
We have discussed blow-up criteria of differential equations from the viewpoint of dynamical systems and singularities.
As a prototype, we have introduced quasi-Poincar\'{e} compactifications, which is a quasi-homogeneous generalization of (homogeneous) Poincar\'{e} compactifications.
Comparing with the other type of compactifications such as Poincar\'{e}-Lyapunov (PL-)disks, we have shown that the qualitative properties of dynamics at infinity are independent of the choice of compactifications, which indicates that we can choose compactifications suitable for our demand (e.g., stationary or periodic blow-up solutions, or their numerical simulations).
\par
We have also shown that the time-variable desingularization and hyperbolic invariant sets, such as equilibria and periodic orbits, at infinity induce blow-up solutions as well as the specific asymptotic behavior including blow-up rates.
These ideas open the door to analysis of blow-ups and related finite-time singularities from algebraic and geometric viewpoints.
\par
We end our arguments by showing several future prospects of the present discussions.

\subsection*{General quasi-homogeneous compactifications}
In the theory of (admissible) homogeneous compactifications \cite{EG2006}, the Poincar\'{e} compactification is settled into the central one in the sense that the direction at infinity is distinguished and that the degeneracy of points at infinity is removed as possible. 
These compactifications give us global coordinate systems compared with directional compactifications such as PL-disks.
On the other hand, the Poincar\'{e}-type compactifications need calculations of radicals, which may break smoothness of (desingularized) vector fields at infinity, as discussed in Section \ref{section-stationary}.
As for homogeneous compactifications, there are several admissible compactifications, such as parabolic compactifications, such that polynomial vector fields are mapped into {\em rational} ones.
They overcome all difficulties coming from radicals appeared in compactifications for studying dynamics, keeping the presence of global coordinate representations. 
\par
With this in mind, it is natural to consider the quasi-homogeneous analogue of general admissible compactifications which avoids difficulties coming from radicals.
Various properties which quasi-Poincar\'{e} compactifications possess in the present arguments will be the basis on constructing an {\em admissible} class of quasi-homogeneous compactifications.
Admissible quasi-homogeneous compactifications which generalize admissible homogeneous compactifications (e.g., \cite{EG2006}) are our next focuses for constructing general theory and applications to, say rigorous numerics of blow-up solutions mentioned below.

\subsection*{General asymptotics of blow-up solutions}
Our main theorems in Section \ref{section-blow-up} show that hyperbolic equilibria and periodic orbits at infinity admitting nontrivial stable manifolds induce blow-up solutions for original systems.
From the viewpoint of dynamical systems, it is natural to consider blow-up solutions whose asymptotics are followed by general hyperbolic (or more general) invariant sets at infinity.
The key issues for these results are topological conjugacy of dynamics to linearized systems and in-phase (or shadowing) property of invariant sets.
Various studies of invariant sets at infinity will lead to new direction of blow-up analysis for differential equations.

\subsection*{Rigorous numerics of blow-up solutions for asymptotically quasi-homogeneous vector fields}
Our present arguments are motivated in {\em rigorous numerics} of trajectories in dynamical systems with certain singularities, such as blow-up solutions, and Riemann solutions admitting singular shocks in systems of conservation laws (e.g., \cite{S2004}).
Rigorous numerics, mainly based on {\em interval arithmetic}, are ones of numerical computation techniques which encloses all numerical errors such as truncation or rounding errors with appropriate mathematical estimates.
Interval arithmetic enables us to compute enclosures where mathematically correct objects are contained in the phase space.
These enclosures give {\em explicit} error bounds of {\em rigorous} solutions. 
\par
Recently, rigorous numerics are applied to validating blow-up solutions of ODEs by the author and his collaborators \cite{TMSTMO}, which applies (homogeneous) compactifications and Lyapunov functions to enclose rigorous blow-up times of blow-up solutions.
One of directions extending arguments in \cite{TMSTMO} is the application of the present theory, which will lead to validate blow-up solutions of (asymptotically) quasi-homogeneous systems, even for periodic blow-ups and singular shock profiles.

\section*{Acknowledgements}
The author was partially supported by Program for Promoting the reform of national universities (Kyushu University), Ministry of Education, Culture, Sports, Science and Technology (MEXT), Japan, and World Premier International Research Center Initiative (WPI), MEXT, Japan.
He also would like to thank Prof. B. Sandstede for giving him information of several preceding works concerning with dynamics at infinity.

\bibliographystyle{plain}
\bibliography{qh_blow_up}

\end{document}